\documentclass[11pt]{amsart}
\usepackage{amsmath,amscd,amssymb,amsfonts}
\usepackage{mathrsfs}
\usepackage{hyperref}
\numberwithin{equation}{section}       
\numberwithin{figure}{section}       

\theoremstyle{plain}

\newtheorem{prop}{Proposition}[section]
\newtheorem{coro}[prop]{Corollary}
\newtheorem{lemm}[prop]{Lemma}

\newtheorem{theoalph}{Theorem}

\newtheorem{coroalph}[theoalph]{Corollary}

\newtheorem{conj}[prop]{Conjecture}

\theoremstyle{definition}
\newtheorem{defi}[prop]{Definition}

\theoremstyle{remark}

\newtheoremstyle{citing}
  {3pt}
  {3pt}
  {\itshape}
  {}
  {\bfseries}
  {.}
  {.5em}
  {\thmnote{#3}}

\theoremstyle{citing}
\newtheorem*{generic}{}

\newcommand{\partn}[1]{{\smallskip \noindent \textbf{#1.}}}

\DeclareMathAlphabet{\mathpzc}{OT1}{pzc}{m}{it} 

\newcommand{\orb}{\textit{O}}

\newcommand{\cO}{\mathcal{O}}

\newcommand{\hDelta}{\widehat{\Delta}}

\newcommand{\tDelta}{\widetilde{\Delta}}

\newcommand{\teta}{\widetilde{\teta}}

\renewcommand{\=}{ : = }

\newcommand{\Ok}{{\cO_k}}

\newcommand{\mk}{\mathfrak{m}_k}
\newcommand{\wtk}{\widetilde{k}}
\newcommand{\wtf}{\widetilde{f}}
\newcommand{\whf}{\widehat{f}}
\newcommand{\whi}{\widehat{i}}
\newcommand{\whg}{\widehat{g}}\newcommand{\whh}{\widehat{h}}

\DeclareMathOperator{\wideg}{wideg}
\DeclareMathOperator{\ord}{ord}

\DeclareMathOperator{\resit}{r\acute{e}sit}

\begin{document}

\title[Generic parabolic points are isolated]{Generic parabolic points are isolated in positive characteristic}

\author[Karl-Olof Lindahl]{Karl-Olof Lindahl$^\dag$}
\thanks{$^\dag$ This research was supported by the by the Royal Swedish Academy of Sciences
and the Swedish Research and Education in Mathematics Foundation SVeFUM}
\address{Karl-Olof Lindahl, Department of Mathematics, Linnaeus University, 351 95, V\"{a}xj\"{o}, Sweden}
\email{karl-olof.lindahl@lnu.se}

\author[Juan Rivera-Letelier]{Juan Rivera-Letelier$^\ddag$}
\thanks{$^\ddag$ This research was partially supported by FONDECYT grant 1141091, Chile}
\address{Juan Rivera-Letelier, Facultad de Matem{\'a}ticas, Pontificia Universidad Cat{\'o}lica de Chile, Avenida Vicu{\~n}a Mackenna~4860, Santiago, Chile}
\email{riveraletelier@mat.puc.cl}
\urladdr{\url{http://rivera-letelier.org/}}

\begin{abstract}
We study analytic germs in one variable having a parabolic fixed point at the
origin, over an ultrametric ground field of positive characteristic.
It is conjectured that for such a germ the origin is isolated as a periodic point.
Our main result is an affirmative solution of this conjecture in the case of a generic
germ with a prescribed multiplier.
The genericity condition is explicit: That the power
series is minimally ramified, \emph{i.e.}, that the degree of the first non-linear term of
each of its iterates is as small as possible.
Our main technical result is a computation of the first significant
terms of a minimally ramified power series.
From this we obtain a lower bound for the norm of nonzero periodic points,
from which we deduce our main result.
As a by-product we give a new and  self-contained proof of a characterization of minimally ramified power series in terms of the iterative residue.
\end{abstract}


\maketitle

%
%

\section{Introduction}
\label{s:introduction}
In this article, we are interested in the dynamics near a parabolic cycle of
analytic maps in positive characteristic.
Recall that a periodic point~$\zeta_0$ of minimal period~$n$ of an
analytic map~$f$ in one variable is \emph{rationally indifferent} or
\emph{parabolic} if~$(f^n)'(\zeta_0)$ is a root of unity, and
it is \emph{irrationally indifferent} if~$(f^n)'(\zeta_0)$ is not a root of unity, but~$|(f^n)'(\zeta_0)| = 1$.

In the complex setting, Yoccoz showed that for an irrationally indifferent cycle of a
quadratic map there is the following dichotomy: Either the map is
locally linearizable near the cycle and then each point in this cycle is
isolated as a periodic point, or every neighborhood of the cycle
contains a cycle of strictly larger minimal period, see~\cite{Yoc95c}.
In contrast, for an ultrametric ground field of characteristic zero only the
first alternative occurs: Every irrationally indifferent cycle is
locally linearizable~\cite{HerYoc83}, and hence every point
in the cycle is isolated as a periodic point.
The case of an ultrametric field of positive characteristic is more
subtle, since irrationally indifferent cycles are usually
not locally linearizable, see for example~\cite[Theorem~$2.3$]{Lin04}
or~\cite[Theorem~$1.1$]{Lin10}.
Nevertheless, every irrationally indifferent
periodic point is isolated as a periodic point~\cite[Corollary~$1.1$]{LinRiv16}.

In this paper we focus on parabolic cycles.
An analytic map is never locally linearizable near a parabolic cycle, except in the trivial case
where an iterate of the map is the identity near the cycle.
Thus, the question that remains is whether (non-trivial) parabolic periodic points are
isolated.
For an ultrametric ground field of characteristic zero the answer is affirmative: The positive residue characteristic case follows
from the fact that periodic points are the zeros of the iterative
logarithm, see~\cite[Proposition~$3.6$]{Rivthese} and
also~\cite{Lub94} for the case where the ground field is discretely
valued; The zero residue characteristic case follows from elementary facts, see for example~\cite[Lemma~$2.1$]{LinRiv16}.

Thus, it only remains the case of parabolic cycles in positive characteristic.
In~\cite{LinRiv16} we proposed the following conjecture.

\begin{conj}[\cite{LinRiv16}, Conjecture~$1.2$]
In positive characteristic, every parabolic periodic point is either isolated as a periodic point, or has a neighborhood on which an iterate of the map is the identity.
\end{conj}

Our main result is a solution of this conjecture for generic parabolic
periodic points.

\begin{generic}[Main Theorem]
In positive characteristic, every generic parabolic periodic point is
isolated as a periodic point.
\end{generic}

The proof of the Main Theorem relies on the connection between the geometric location of
periodic points of power series with integer coefficients, and the
lower ramification numbers of wildly ramified field automorphisms that was established in~\cite{LinRiv16}.
Lower ramification numbers of wildly ramified field automorphisms have
previously been studied by Sen~\cite{Sen69}, Keating~\cite{Kea92},
Laubie and Sa{\"{\i}}ne~\cite{LauSai98}, Wintenberger~\cite{Win04},
among others.

We now proceed to describe the content and proof of the Main Theorem in more precise terms.
Replacing the map by an iterate if necessary, we restrict to the case
of fixed points.
Moreover, after conjugating by a translation we assume the fixed point is
the origin.
So, from now on we restrict to study power series~$f(\zeta)$ such
that~$f(0) = 0$ and such that~$f'(0)$ is a root of unity.
We call such a power series \emph{parabolic}.
Finally, after a scale change we can restrict to the case of power
series with integer coefficients.
In the case where the ground field is algebraically closed, this last
condition is equivalent to the condition that the power series is
convergent and univalent on the open unit disk, see for example~\cite[\S$1.3$]{Rivthese}.
So these normalizations are analogous to those used in the complex
setting by Yoccoz in~\cite{Yoc95c}.

The genericity condition in the Main Theorem is described explicitly in terms of
the lower ramification numbers of the map: It requires that the lower
ramification numbers are as small as possible.
Such power series are called \emph{minimally ramified},
see~\S\ref{ss:minimally ramified are generic} for precisions.
So the Main Theorem is a direct consequence of~$2$ independent facts:
\begin{enumerate}
\item[A.]
Minimally ramified parabolic power series are generic
(Theorem~\ref{t:genericity} in~\S\ref{ss:minimally ramified are generic});
\item[B.]
For a parabolic power series that is minimally ramified, the origin is
isolated as a periodic point (Corollary~\ref{c:isolated} in~\S\ref{ss:periodic of minimally ramified}).
\end{enumerate}
Minimally ramified power series appear naturally when studying ``optimal cycles'' of irrationally
indifferent periodic points, see~\cite{LinRiv16}.
They were first introduced in a more restricted setting
by Laubie, Movahhedi, and Salinier in~\cite{LauMovSal02}, in
connection with Lubin's conjecture in~\cite{Lub94}.

Our main technical result is a calculation of the first significant
terms of the iterates of power series in a certain normal form (Main Lemma in~\S\ref{s:a reduction}).
This allows us to give a new and self-contained proof of the
characterization of minimally ramified power series given
in~\cite{LinRiv16} (Theorem~\ref{t:characterization
  of minimally ramified} in~\S\ref{ss:organization}).

We proceed to describe our main results in more detail.
Throughout the rest of the introduction we fix a prime number~$p$ and
a field~$k$ of characteristic~$p$.

\subsection{Minimally ramified power series are generic}
\label{ss:minimally ramified are generic}
Denote by~$\ord(\cdot)$ the valuation on~$k[[\zeta]]$ defined for a nonzero power series as the 
lowest degree of its nonzero terms, and for the zero power series~$0$ by~$\ord(0) = + \infty$.

\begin{defi}
\label{d:ramification number}
For a power series~$g(\zeta)$ in~$k[[\zeta]]$
satisfying~$g(0) = 0$ and~$g'(0) = 1$, define for each integer~$n \ge 0$
$$ i_n(g) \= \ord \left( \frac{g^{p^n}(z) - z}{z} \right). $$
\end{defi}

If~$g(\zeta)$ is as in the definition and~$n \ge 1$ is such that~$i_n(g)$
is finite, then~$i_0(g)$, \ldots, $i_{n - 1}(g)$ are all finite and we
have
\begin{displaymath}
  i_0(g) < i_1(g) < \cdots < i_n(g),
\end{displaymath}
see for example~\cite[Lemma~$3.6$]{LinRiv16}.

Let~$f(\zeta)$ be a parabolic power series in~$k[[\zeta]]$ and denote by~$q$
the order of~$f'(0)$, so that~$f^q(0) = 0$ and~$(f^q)'(0) = 1$.
Then for every integer~$n \ge 0$ we have
\begin{equation}
  \label{e:least ramification}
  i_n(f^q) \ge q \frac{p^{n + 1} - 1}{p - 1},
\end{equation}
see~\cite[Proposition~$3.2$]{LinRiv16}.
This motivates the following definition.

\begin{defi}
Let~$f(\zeta)$ be a parabolic power series in~$k[[\zeta]]$ and denote
by~$q$ the order of~$f'(0)$.
Then~$f$ is \emph{minimally ramified} if equality holds
in~\eqref{e:least ramification} for every~$n \ge 0$.
\end{defi}

Roughly speaking, the following result asserts that among parabolic power
series with a prescribed multiplier, those that are minimally ramified are generic.

\begin{theoalph}
\label{t:genericity}
Let~$p$ be a prime number, $k$ a field of characteristic~$p$, and~$F$ the prime field of~$k$.
Fix a root of unity~$\gamma$ in~$k$ and denote by~$q$ its order.
Then there is a nonzero polynomial~$M_q(x_1,\ldots,x_{2q})$ with
coefficients in~$F(\gamma)$, such that a power series~$f(\zeta)$ in~$k[[\zeta]]$ of the form
\begin{equation}
  \label{e:typical series}
f(\zeta)
=
\gamma \zeta \left( 1 + \sum_{i = 1}^{+ \infty} c_i \zeta^i \right)  
\end{equation}
is minimally ramified if and only if~$M_q(c_1, \ldots, c_{2q}) \neq 0$.
\end{theoalph}

\subsection{Periodic points of minimally ramified parabolic power series}
\label{ss:periodic of minimally ramified}
Now we turn to periodic points of parabolic power series that are
minimally ramified.
For such a power series that has integer coefficients, we estimate from below the norm of nonzero periodic points.
The estimate is in terms of some invariants that we proceed
to describe.

In this section we fix a parabolic power series~$f(\zeta)$
in~$k[[\zeta]]$ and denote by~$q$ the order of~$\gamma \= f'(0)$.
We assume~$i_0(f^q) = q$, which is weaker than~$f$ being minimally ramified.

The first invariant is the coefficient~$\delta_0(f^q)$ of~$\zeta^{q + 1}$
in~$f^q$ so that $\delta_0(f^q)\neq 0$ and that
$$ f^q(\zeta)
=
\zeta \left( 1 + \delta_0(f^q) \zeta^q + \cdots \right). $$
The coefficient~$\delta_0(f^q)$ is invariant under conjugacy by power series
that are tangent to the identity at~$\zeta =
0$ (Lemma~\ref{l:quasi-invariants}).

The second invariant is the ``iterative residue'' introduced
in~\cite[\S$4$]{LinRiv16}.
It is a positive characteristic analog of the invariant introduced
in the complex setting by {\'E}calle in~\cite{Eca75}.
To define it, suppose first~$\gamma = 1$, and let~$a_1$ and~$a_2$
in~$k$ be
such that~$f(\zeta)$ is of the form
$$ f(\zeta)
=
\zeta \left( 1 + a_1 \zeta + a_2 \zeta^2 + \cdots \right). $$
Our hypothesis~$i_0(f) = 1$ implies~$a_1 \neq
0$.
The \emph{iterative residue} $\resit(f)$ of~$f$ is defined by
$$ \resit(f) \= 1 - \frac{a_2}{a_1^2}. $$
Suppose now that~$\gamma \neq 1$, so~$q \ge 2$.
Then~$f$ is conjugated to a power series of the form
\begin{equation}
  \label{e:reduced form}
g(\zeta)
=
\gamma \zeta \left( 1 + \sum_{j = 1}^{+ \infty} a_j \zeta^{jq} \right),
\end{equation}
see Lemma~\ref{l:basic normal form reloaded} or~\cite[Proposition~$4.1$]{LinRiv16}.
Our hypothesis~$i_0(f^q) = q$ implies~$a_1 \neq 0$
(Lemma~\ref{l:quasi-invariants} and Corollary~\ref{c:minimally ramified q series} or~\cite[Proposition~$4.1$]{LinRiv16}).
The \emph{iterative residue} $\resit(f)$ of~$f$ is defined by
$$ \resit(f) \= \frac{q + 1}{2} - \frac{a_2}{a_1^2};
\footnote{Note that in the case~$p = 2$ the integer~$q$ is odd,
  so~$\frac{q + 1}{2}$ defines an element of~$k$.} $$
It only depends on~$f$ and not on~$g$.
In all the cases~$\resit(f)$ is a conjugacy invariant of~$f$, see~\cite[Lemma~$4.3$]{LinRiv16}.

Suppose~$k$ is endowed with a norm~$|\cdot|$ and denote by~$\Ok$
the ring of integers of~$(k, |\cdot|)$ and by~$\mk$ its maximal ideal.
Then the minimal period of each periodic point of~$f$ in~$\mk \setminus \{ 0 \}$ is of the form~$qp^n$, for some integer~$n \ge 0$, 
see for example~\cite[Lemma~$2.1$]{LinRiv16}.

\begin{theoalph}
\label{t:isolated}
Let~$p$ be a prime number and~$(k, |\cdot|)$ an ultrametric field of characteristic~$p$.
Let~$f(\zeta)$ be a parabolic power series in~$\Ok[[\zeta]]$ and denote
by~$q$ the order of~$f'(0)$.
Then for every integer~$n\geq 1$ and every periodic point~$\zeta_0$ of~$f$ of minimal period~$qp^n$, we have
\begin{equation*}
|\zeta_0|
\ge
\begin{cases}
\left| \delta_0(f^q) \right|^{\frac{1}{q}}
\cdot \left|\resit(f) \right|^{\frac{1}{qp}}
& \text{if~$p$ is odd}
\\ &
\text{or if~$p = 2$ and~$n = 1$;}
\\
\left| \delta_0(f^q) \right|^{\frac{1}{q}}
\cdot \left| \resit(f) \left( 1 - \resit(f) \right) \right|^{\frac{1}{4q}}
& \text{if~$p = 2$ and~$n \ge 2$.}
\end{cases}
\end{equation*}
\end{theoalph}

In the case~$f$ is minimally ramified, we have~$\resit(f) \neq 0$, and also~$\resit(f) \neq 1$ if~$p = 2$, see~\cite[Theorem~E]{LinRiv16} or Theorem~\ref{t:characterization of minimally ramified} in~\S\ref{ss:organization}.
So the following corollary is a direct consequence of the previous theorem and of~\cite[Lemma~$2.1$]{LinRiv16}.
\begin{coroalph}
\label{c:isolated}
Let~$(k, |\cdot|)$ be an ultrametric field of positive characteristic.
Then for each parabolic power series in~$\Ok[[\zeta]]$ that is minimally ramified, the origin is isolated as a periodic point.
\end{coroalph}
\subsection{Strategy and organization}
\label{ss:organization}
Our main technical result is a calculation of the first significant
terms of the iterates of a power series in~$k[[\zeta]]$ of the
form~\eqref{e:reduced form}.
This is stated as the Main Lemma in~\S\ref{s:a reduction}.
A direct consequence is an explicit condition on the coefficients~$a_1$
and~$a_2$ for the power series~$g(\zeta)$ to be minimally ramified
(Corollary~\ref{c:minimally ramified q series}).
Theorem~\ref{t:genericity} is obtained from this result using that
every parabolic power series is conjugated to a power series of the
form~\eqref{e:reduced form} by an invertible map in~$\Ok[[\zeta]]$ (Lemma~\ref{l:basic normal form reloaded}).
The same strategy allows us to give a new and self-contained proof of the
following characterization of minimally ramified power series,\footnote{The proof of this result in~\cite{LinRiv16} relies
on~\cite[Corollary~$1$]{LauSai98}; the Main Lemma allows us to give a direct proof that avoids this last result.}
see~\S\ref{ss:proof of genericity and isolated}.
\begin{theoalph}[\cite{LinRiv16}, Theorem~E]
\label{t:characterization of minimally ramified}
Let~$p$ be a prime number and~$k$ a field of characteristic~$p$.
Moreover, let~$f(\zeta)$ be a parabolic power series in~$k[[\zeta]]$,
and denote by~$q$ the order of~$f'(0)$.
If~$p$ is odd (resp. $p = 2$), then~$f$ is minimally ramified if and
only if
$$ i_0(f^q) = q
\text{ and }
\resit(f) \neq 0 $$
$$ \left( \text{resp. } i_0(f^q) = q,
\resit(f) \neq 0,
\text{ and }
\resit(f) \neq 1 
\right). $$
\end{theoalph}

To prove Theorem~\ref{t:isolated}, we first prove a version for
parabolic power series of the Period Points Lower Bound
of~\cite{LinRiv16} (Lemma~\ref{l:periodic bound}).
Combined with the Main Lemma in~\S\ref{s:a reduction} this yields a lower bound for the norm of
the periodic points of a power series of the form~\eqref{e:reduced form} (Corollary~\ref{c:generic periodic points lower bound g q}).
Theorem~\ref{t:isolated} is then obtained using again that every parabolic
power series is conjugated to a power series of the
form~\eqref{e:reduced form} by an invertible map in~$\Ok[[\zeta]]$ (Lemma~\ref{l:basic normal form reloaded}).

After some preliminaries in~\S\ref{s:periodic of parabolic}, we state the
Main Lemma and deduce its Corollaries~\ref{c:minimally ramified q
  series} and~\ref{c:generic periodic points lower bound g q} in the first part of~\S\ref{s:a reduction}.
The proofs of Theorems~\ref{t:genericity}, \ref{t:isolated},
and~\ref{t:characterization of minimally ramified} assuming the Main Lemma are given in~\S\ref{ss:proof of genericity and isolated}.
Finally, the proof of the Main Lemma is given in~\S\ref{s:proof of the
  main lemma}.
As mentioned above, the Main Theorem is a direct consequence of
Theorems~\ref{t:genericity} and~\ref{t:isolated}.

\subsection{Acknowledgement}
We would like to thank the referee for comments and corrections that helped us improve the presentation.

\section{Periodic points of parabolic power series}
\label{s:periodic of parabolic}

After some preliminaries in~\S\ref{ss:preliminaries},
in~\S\ref{ss:periodic of parabolic} we give some basic facts about periodic
points of parabolic power series with integer coefficients.
In particular, we give a general lower bound for the distance to the
origin of a periodic point (Lemma~\ref{l:periodic bound}) that is used in the proof Theorem~\ref{t:isolated}.

\subsection{Preliminaries}
\label{ss:preliminaries}
Given a ring~$R$ and an element~$a$ of~$R$, we denote by~$\langle a \rangle$ the ideal of~$R$ generated by~$a$.

Let~$(k, | \cdot |)$ be an ultrametric field.
Denote by~$\Ok$ the ring of integers of~$k$, by~$\mk$ the maximal ideal of~$\Ok$, and by~$\wtk \= \Ok / \mk$ the residue field of~$k$.
Moreover, denote the projection in~$\wtk$ of an element~$a$ of~$\Ok$ by~$\widetilde{a}$; it is the \emph{reduction of~$a$}.
The \emph{reduction} of a power series~$f(\zeta)$ in~$\Ok[[\zeta]]$ is the power series~$\widetilde{f}(z)$ in~$\wtk[[z]]$ whose coefficients are the reductions of the corresponding coefficients of~$f$.

For a power series~$f(\zeta)$ in~$\Ok[[\zeta]]$, the \emph{Weierstrass degree~$\wideg(f)$ of~$f$} is the order in~$\wtk[[z]]$ of the reduction~$\wtf(z)$ of~$f(\zeta)$.
When~$\wideg(f)$ is finite, the number of zeros of~$f$ in~$\mk$,
counted with multiplicity, is less than or equal to~$\wideg(f)$; see for example~\cite[{\S}VI, Theorem~$9.2$]{Lan02}.

A power series~$f(\zeta)$ in~$\Ok[[\zeta]]$ converges in~$\mk$.
If in addition~$|f(0)| < 1$, then by the ultrametric inequality~$f$ maps~$\mk$ to itself.
In this case a point~$\zeta_0$ in~$\mk$ is \emph{periodic for~$f$}, if there is an integer~$n \ge 1$ such that~$f^n(\zeta_0) = \zeta_0$.
In this case \emph{$\zeta_0$ is of period~$n$}, and~$n$ is a \emph{period of~$\zeta_0$}.
If in addition~$n$ is the least integer with this property, then~$n$ is the \emph{minimal period of~$\zeta_0$} and~$(f^n)'(\zeta_0)$ is the \emph{multiplier of~$\zeta_0$}.
Note that an integer~$n \ge 1$ is a period of~$\zeta_0$ if and only if it is divisible by the minimal period of~$\zeta_0$.

The following definition is consistent with the
definition of~$\delta_0(\cdot)$ in the introduction.

\begin{defi}
\label{d:multiplicity}
Let~$p$ be a prime number and~$k$ field of characteristic~$p$.
For a power series~$g(\zeta)$ in~$k[[\zeta]]$ satisfying~$g(0) = 0$
and~$g'(0) = 1$, define for each integer~$n \ge 0$ the element~$\delta_n(g)$ of~$k$ as follows:
Put~$\delta_n(g) \= 0$ if~$i_n(g) = + \infty$, and otherwise
let~$\delta_n(g)$ be the coefficient
of~$\zeta^{i_n(g)+1}$ in the power series~$g^{p^n}(\zeta) - \zeta$.
\end{defi}

Note that in the case where~$i_n(g)$ is finite, $\delta_n(g)$ is
nonzero.
Moreover, if~$k$ is an ultrametric field and~$g(\zeta)$ is in
$\Ok[[\zeta]]$, then~$g^{p^n}(\zeta) - \zeta$ is also
in~$\Ok[[\zeta]]$, and therefore~$\delta_n(g)$ is in~$\Ok$.
\begin{lemm}
  \label{l:quasi-invariants}
Let~$p$ be a prime number and~$k$ field of characteristic~$p$.
Moreover, let~$f(\zeta)$ and~$\whf(\zeta)$ be parabolic power series
in~$k[[\zeta]]$ and denote by~$q$ the order of~$f'(0)$.
Suppose there is a power series~$h(\zeta)$ in~$k[[\zeta]]$ such that
$$ h(0) = 0,
h'(0) \neq 0,
\text{ and }
f \circ h = h \circ \whf. $$
Then~$\whf'(0) = f'(0)$ and for every integer~$n \ge 0$ we have
$$ i_n(\whf^q) = i_n(f^q)
\text{ and }
\delta_n(\whf^q) = (h'(0))^{i_n(f^q)}\cdot \delta_n(f^q). $$
\end{lemm}
\begin{proof}
From~$f \circ h = h \circ \whf$ we have~$f'(0) h'(0) = h'(0)
\whf'(0)$.
Together with our assumption~$h'(0) \neq 0$ this implies~$\whf'(0) = f'(0)$.

Fix an integer~$n \ge 0$, and note that~$f^{qp^n} \circ h = h \circ
\whf^{qp^n}$.
If~$i_n(f) = + \infty$, then~$\delta_n(f^{qp^n}) = 0$ and~$f^{qp^n}(\zeta)
= \zeta$.
This implies~$\whf^{qp^n}(\zeta) = \zeta$, so~$i_n(\whf^q) = + \infty$
and~$\delta_n(\whf^q) = 0$.
This proves the lemma when~$i_n(f) = + \infty$.
Interchanging the roles of~$f$ and~$\whf$, this also proves the lemma
when~$i_n(\whf) = + \infty$.
Suppose~$i_n(f)$ and~$i_n(\whf)$ are finite, and put
$$ i_n \= i_n(f^q), \text{ }
\whi_n \= i_n(\whf^q),
\text{ and }
h(\zeta) = \zeta \cdot \sum_{i = 0}^{+ \infty} c_i \zeta^i. $$
Then we have
  \begin{multline*}
  \begin{aligned}
f^{qp^n} \circ h(\zeta)
& \equiv
\zeta \left( c_0 + c_1 \zeta + \cdots + c_{i_n} \zeta^{i_n} \right)
\left( 1 + \delta_n(f^q) (c_0 \zeta)^{i_n} \right)
\mod \left\langle \zeta^{i_n + 2} \right\rangle
\\ & \equiv
\zeta \left( c_0 + c_1 \zeta + \cdots + c_{i_n - 1} \zeta^{i_n - 1}
+ \left( c_{i_n} + c_0^{i_n + 1} \delta_n(f^q) \right) \zeta^{i_n} \right)
  \end{aligned}
\\
\mod \left\langle \zeta^{i_n + 2} \right\rangle.    
\end{multline*}
On the other hand,
\begin{multline*}
  \begin{aligned}
    h \circ \whf^q (\zeta)
& \equiv
\zeta \left(1 + \delta_n(\whf^q) \zeta^{\whi_n} \right)
\left( c_0 + c_1 \zeta + \cdots + c_{\whi_n} \zeta^{\whi_n} \right)
\mod \left\langle \zeta^{\whi_n + 2} \right\rangle.
\\ & \equiv
\zeta \left( c_0 + c_1 \zeta + \cdots + c_{\whi_n - 1} \zeta^{\whi_n - 1} +
  \left( c_{\whi_n} + c_0 \delta_n(\whf^q) \right) \zeta^{\whi_n} \right)
  \end{aligned}
\\
\mod \left\langle \zeta^{\whi_n + 2} \right\rangle.
\end{multline*}
Comparing coefficients and using that~$c_0$, $\delta_n(f^q)$,
and~$\delta_n(\whf^q)$ are all different from zero, we conclude
that~$i_n = \whi_n$ and that~$\delta_n(\whf^q) = c_0^{i_n}
\delta_n(f^q)$, as wanted.
\end{proof}

\subsection{Periodic points of parabolic power series}
\label{ss:periodic of parabolic}
The following lemma is well-known, see for
example~\cite[Lemma~$2.1$]{LinRiv16} for a proof.
\begin{lemm}
\label{l:admissible periods}
Let~$p$ be a prime number and~$k$ an ultrametric field of
characteristic~$p$.
Moreover, let~$f(\zeta)$ be a parabolic power series in~$k[[\zeta]]$
and denote by~$q$ the order of~$f'(0)$.
Then~$q$ is not divisible by~$p$, and the minimal period of each
periodic point of~$f$ is of the form~$q p^n$ for some integer~$n \ge 0$.
\end{lemm}

The following lemma is a version of~\cite[Lemma~$2.3$]{LinRiv16} for
parabolic power series, with a similar proof.
We have restricted to ground fields of positive characteristic for simplicity.
It is one of the ingredients in the proof of Theorem~\ref{t:isolated}.
Before stating the lemma we recall that for a power series~$f(\zeta)$ in~$\Ok[[\zeta]]$, 
the Weierstrass degree~$\wideg(f)$ of~$f$ is the order in~$\wtk[[z]]$ of the reduction~$\wtf(z)$ of~$f(\zeta)$.

\begin{lemm}[Periodic points lower bound for parabolic series]
\label{l:periodic bound}
Let~$p$ be a prime number and~$(k, | \cdot |)$ an ultrametric field of characteristic~$p$.
Moreover, let~$f(\zeta)$ be a parabolic power series in~$\Ok[[\zeta]]$ and denote by~$q$ the order of~$f'(0)$.
Then the following properties hold.
\begin{enumerate}
\item[1.]
Let~$w_0$ 
in $\mathfrak{m}_k$      
be a periodic point of~$f$ of minimal period~$q$.
In the case~$q = 1$, assume~$w_0 \neq 0$.
Then we have
\begin{equation}
\label{e:fixed bound}
|w_0| \ge |\delta_{0}(f^q)|^{\frac{1}{q}},
\end{equation}
with equality if and only
if
\begin{equation}
\label{e:minimal fixed wideg increment}
  \wideg \left( f^q(\zeta) - \zeta \right) = i_0(f^q)+q + 1.
\end{equation}
Moreover, if (\ref{e:minimal fixed wideg increment}) holds, then the cycle containing~$w_0$ is the only cycle of 
minimal period~$q$ of~$f$ in~$\mk \setminus \{ 0 \}$, and for every point~$w_0'$ in 
this cycle~$|w_0'|=|\delta_0(f^q)|^{\frac{1}{q}}$.
\item[2.]
Let~$n \ge 1$ be an integer and~$\zeta_0$ 
in $\mathfrak{m}_k$ 
a periodic point of~$f$ of minimal period~$q p^n$.
If in addition~$i_n(f^q)< + \infty$, then we have
\begin{equation}
\label{e:periodic bound}
|\zeta_0|
\ge
\left| \frac{\delta_n(f^q)}{\delta_{n-1}(f^q)} \right|^{\frac{1}{q p^n}},
\end{equation}
with equality 
if and only if
\begin{equation}
\label{e:minimal wideg increment}
\wideg \left( \frac{f^{q p^n}(\zeta) - \zeta}{f^{q p^{n - 1}}(\zeta) - \zeta} \right)
=
i_n(f^q)-i_{n-1}(f^q)+q p^n.
\end{equation}
Moreover, if (\ref{e:minimal wideg increment}) holds, 
then the cycle containing~$\zeta_0$ is the only cycle of minimal period~$qp^n$ of~$f$ in $\mathfrak{m}_k$, 
and for every point~$\zeta_0'$ in this cycle 
$|\zeta_0'|
=
\left| \frac{\delta_n(f^q)}{\delta_{n-1}(f^q)} \right|^{\frac{1}{q p^n}}$.
\end{enumerate}
\end{lemm}

The proof of this lemma is below, after the following lemmas.

\begin{lemm}{\cite[Lemma~$2.2$]{LinRiv16}}
\label{l:fixed are periodic}
Let~$(k, | \cdot |)$ be a complete ultrametric field and~$g(\zeta)$ a power series in~$\Ok[[\zeta]]$ such that~$|g(0)| < 1$.
Then for each integer~$m \ge 1$ the power series~$g(\zeta) - \zeta$ divides~$g^m(\zeta) - \zeta$ in~$\Ok[[\zeta]]$.
\end{lemm}

\begin{lemm}
  \label{l:elementary division}
Let~$k$ be a complete ultrametric field and let~$h(\zeta)$ be a power series in~$\Ok[[\zeta]]$.
If~$\xi$ is a zero of~$h$ in~$\mk$, then~$\zeta - \xi$ divides~$h(\zeta)$ in~$\Ok[[\zeta]]$.
\end{lemm}
\begin{proof}
Put~$T(\zeta) = \zeta + \xi$ and note that~$h \circ T(\zeta)$ vanishes at~$\zeta = 0$ and is in~$\Ok[[\zeta]]$.
This implies that~$\zeta$ divides~$h \circ T(\zeta)$ in~$\Ok[[\zeta]]$.
Letting~$g(\zeta) \= h \circ T(\zeta) / \zeta$, it follows that the power series~$g \circ T^{-1}(\zeta) = h(\zeta) / (\zeta - \xi)$ is in~$\Ok[[\zeta]]$, as wanted.
\end{proof}

\begin{proof}[Proof of Lemma~\ref{l:periodic bound}]
Replacing~$k$ by one of its completions if necessary, assume~$k$ complete.

We use the fact that, since~$|f'(0)| = 1$, the power series~$f$ maps~$\mk$ to itself isometrically, 
see for example~\cite[\S$1.3$]{Rivthese}. 
We also use several times that when~$\wideg(f)$ is finite, 
we have:
\begin{enumerate}
\item[I.] The Weierstrass degree~$\wideg(f)$ of~$f$ equals the degree of the 
lowest degree term of $f$ whose coefficient is of norm equal to $1$;
\item[II.] The number of zeros of~$f$ in~$\mk$,
counted with multiplicity, is less than or equal to~$\wideg(f)$. 
\end{enumerate}

\partn{1}
To prove~\eqref{e:fixed bound}, let~$w_0$ 
in $\mathfrak{m}_k$  
be a periodic point of~$f$ of minimal period~$q$.
Note that every point in the forward orbit~$\orb$ of~$w_0$ under~$f$ is a zero of the power series~$(f^q(\zeta) - \zeta) / \zeta$, 
and that the coefficient of the lowest degree term of this power series is~$\delta_0(f^q)$.
On the other hand,~$\orb$ consists of~$q$ points, and, since~$f$ maps~$\mk$ to itself isometrically, 
all the points in~$\orb$ have the same norm.
Applying Lemma~\ref{l:elementary division} inductively with~$\xi$ replaced by each element of~$\orb$, 
it follows that~$\prod_{w_0' \in \orb} (\zeta - w_0')$ divides~$f^q(\zeta) - \zeta$ in~$\Ok[[\zeta]]$.
That is, the power series
\begin{equation}
\label{e:reduced fixed points equation}
\frac{f^q(\zeta) - \zeta}{\zeta} / \prod_{w_0' \in \orb} (\zeta -
w_0')
\end{equation}
is in~$\Ok[[\zeta]]$.
Note that the lowest degree term of this series is of degree~$i_0(f^q)$ and its coefficient is
\begin{equation}
\label{e:total product of fixed points}
\delta_0(f^q) / \prod_{w_0' \in \orb} (- w_0');
\end{equation}
which is therefore in~$\Ok$.
We thus have
\begin{equation}
\label{e:powered fixed lower bound}
|w_0|^q = \prod_{w_0' \in \orb} |w_0'|
\ge
\left|\delta_0(f^q)\right|,
\end{equation}
and therefore~\eqref{e:fixed bound}.
Moreover, equality holds precisely when 
the coefficient~\eqref{e:total product of fixed points} of the lowest degree 
term 
of~\eqref{e:reduced fixed points equation} has norm equal to~$1$.
Thus, in view of Fact I stated at the beginning of the proof,
it follows that equality in~\eqref{e:powered fixed lower bound}, and hence
in~\eqref{e:fixed bound}, holds precisely when the Weierstrass degree
of~\eqref{e:reduced fixed points equation} is equal to~$i_0(f^q)$.
On the other hand, the cardinality of~$\orb$ is equal to~$q$, and
therefore the Weierstrass degree of~\eqref{e:reduced fixed points
  equation} is equal to~$\wideg(f^q(\zeta) - \zeta) - q - 1$.
Combining these facts we conclude that equality in~\eqref{e:fixed
  bound} holds if and only if we have~\eqref{e:minimal fixed wideg increment}.
Finally,
using the cardinality of~$\orb$ together with Facts~I and~II above, 
when this last equality holds, the set~$\orb$ is the set of all zeros of~$(f^q(\zeta) - \zeta)/\zeta$ in~$\mk$, 
so~$\orb$ is the only cycle of minimal period~$q$ of~$f$.
This completes the proof of part~$1$.

\partn{2}
To prove~\eqref{e:periodic bound}, let~$n \ge 1$ be an integer such
that~$i_{n-1}(f^q) < + \infty$, and~$\zeta_0$ in $\mathfrak{m}_k$ a periodic point of~$f$ of minimal period~$qp^n$.
By Lemma~\ref{l:fixed are periodic} with~$g = f^{q p^{n - 1}}$ and~$m = p$, the power series~$f^{q p^{n - 1}}(\zeta) - \zeta$ 
divides~$f^{q p^n}(\zeta) - \zeta$ in~$\Ok[[\zeta]]$.
Note that every point in the forward orbit~$\orb$ of~$\zeta_0$ 
in $\mathfrak{m}_k$ 
under~$f$ is a zero of the power series
$$ h(\zeta) \= \frac{f^{q p^n}(\zeta) - \zeta}{f^{q p^{n - 1}}(\zeta) - \zeta}, $$
and that the lowest degree term of this power series is of
degree~$i_n(f^q) - i_{n - 1}(f^q)$ and its coefficient is equal to
$$\delta(h)\=\frac{\delta_{n}(f^q)}{\delta_{n-1}(f^q)}.$$
On the other hand,~$\orb$ consists of~$q p^n$ points, and, since~$f$ maps~$\mk$ to itself isometrically, 
all the points in~$\orb$ have the same norm.
Applying Lemma~\ref{l:elementary division} inductively with~$\xi$ replaced by each element of~$\orb$, 
it follows that~$\prod_{\zeta_0' \in \orb} (\zeta - \zeta_0')$ divides~$h(\zeta)$ in~$\Ok[[\zeta]]$.
In particular, the power series~$h(\zeta) / \prod_{\zeta_0' \in \orb}
(\zeta - \zeta_0')$ is in~$\Ok[[\zeta]]$, so the coefficient of its lowest degree term,
\begin{equation}
\label{e:total product of periodic points}
\left( \frac{\delta_{n}(f^q)}{\delta_{n-1}(f^q)} \right) / \prod_{\zeta_0' \in \orb} (- \zeta_0'),
\end{equation}
is in~$\Ok$.
We thus have
\begin{equation}
\label{e:powered lower bound}
|\zeta_0|^{q p^n} = \prod_{\zeta_0' \in \orb} |\zeta_0'|
\ge
\left| \frac{\delta_{n}(f^q)}{\delta_{n-1}(f^q)} \right|,
\end{equation}
and therefore~\eqref{e:periodic bound}.
Note that equality holds if and only if the lowest degree coefficient~\eqref{e:total product of periodic points} 
of the power series~$h(\zeta) / \prod_{\zeta_0' \in \orb} (\zeta - \zeta_0')$ has norm equal to~$1$.
In view of Fact~I stated at the beginning of the proof, we conclude
that equality in~\eqref{e:powered lower bound}, and hence
in~\eqref{e:periodic bound}, holds if and only if
$$ \wideg \left(h(\zeta) / \prod_{\zeta_0' \in \orb} (\zeta - \zeta_0')
\right)
=
i_n(f^q)-i_{n-1}(f^q). $$
On the other hand, the cardinality of~$\orb$ is equal to~$qp^n$, and
therefore we have
$$ \wideg \left(h(\zeta) / \prod_{\zeta_0' \in \orb} (\zeta - \zeta_0') \right)
=
\wideg \left( \frac{f^{q p^n}(\zeta) - \zeta}{f^{q p^{n - 1}}(\zeta) -
    \zeta} \right) - q p^n. $$
Combining these facts,
we conclude that equality holds in~\eqref{e:periodic bound} if and only~\eqref{e:minimal wideg increment}.
Finally,
using the cardinality of~$\orb$ together with Facts~I and~II, 
when this last equality holds~$\orb$ is the set of all zeros of~$\frac{f^{q p^n}(\zeta) - \zeta}{f^{q p^{n - 1}}(\zeta) - \zeta}$ in~$\mk$, 
so~$\orb$ is the only cycle of minimal period~$qp^n$ of~$f$.
This completes the proof of part~$2$, and of the lemma.
\end{proof}

\section{A reduction}
\label{s:a reduction}
In this section we prove Theorems~\ref{t:genericity},
\ref{t:isolated}, and~\ref{t:characterization of minimally ramified} assuming the following result, which is proved 
in \S\ref{s:proof of the main lemma}.
Note that for an odd integer~$q$, each of the numbers~$\frac{q - 1}{2}$,
$\frac{q + 1}{2}$, and~$\frac{q^2 - 1}{2}$ is an integer, and
therefore defines an element on each field of
characteristic~$2$. 

\begin{generic}[Main Lemma]
 Let~$p$ be a prime number, $k$~a field of characteristic~$p$, 
 and~$q\geq 1$ an integer that is not divisible by~$p$.
Given~$\gamma$ in~$k$ such that~$\gamma^q=1$, let~$g(\zeta)$ be a power series in~$k[[\zeta]]$ of the form
\begin{equation}
\label{e:reduced form reloaded}
g(\zeta)
=
\gamma\zeta \left(1 + \sum_{j=1}^{+\infty}a_j\zeta^{jq}\right).
\end{equation}
Then we have
\begin{equation}
\label{e:generic form g after q iterates}
g^q(\zeta)
\\ \equiv
\zeta \left( 1 + qa_1 \zeta^q +
  q\left(\left(\frac{q^2 - 1}{2} \right)a_1^2+ a_2\right) \zeta^{2q}
\right)
\mod \left\langle \zeta^{3q + 1} \right\rangle.
\end{equation}
Moreover, if for a given integer~$n \ge 1$ we put
\begin{align*}
\chi_{q, n}
& \=
\begin{cases}
q a_1^{p^n-\frac{p^n-1}{p-1}}
\left(\frac{q+1}{2}a_1^2-a_2\right)^\frac{p^n-1}{p-1}
& \text{if~$p$ is odd;}
\\
a_1 \left(\frac{q+1}{2}a_1^2-a_2\right)
& \text{if~$p =2$ and~$n = 1$;}
\\
a_1 \left(\frac{q - 1}{2} a_1^2 - a_2 \right)^{2^{n-1}-1}
\left(\frac{q+1}{2}a_1^2-a_2\right)^{2^{n-1}}
& \text{if~$p =2$ and~$n \ge 2$,}
\end{cases}
\\ \intertext{and}
\xi_{q, n}
& \=
\begin{cases}
- q a_1^{p^n-\frac{p^n-1}{p-1}-1}
\left(\frac{q+1}{2}a_1^2-a_2\right)^{\frac{p^n-1}{p-1}+1}
& \text{if~$p$ is odd;}
\\
a_2^{2^{n-1}} \left(a_1^2-a_2\right)^{2^{n-1}}
& \text{if~$p = 2$,}
\end{cases}
\end{align*}
then we have
\begin{multline}
\label{e:generic form g iterates}
g^{qp^n}(\zeta)
\equiv
\zeta \left( 1 + \chi_{q, n}\zeta^{q\frac{p^{n+1}-1}{p-1}} + \xi_{q, n}
  \zeta^{q\frac{p^{n+1}-1}{p-1}+q} \right)
\\
\mod \left\langle \zeta^{q\frac{p^{n+1}-1}{p-1} + 2q + 1} \right\rangle.
\end{multline}
\end{generic}

The proof of the Main Lemma is given in~\S\ref{s:proof of the main lemma}.
We now proceed to state some corollaries of this result that are used
in the proofs of Theorems~\ref{t:genericity}, \ref{t:isolated} and~\ref{t:characterization of minimally ramified}, which are given in~\S\ref{ss:proof of genericity and isolated}.

\begin{coro}
\label{c:minimally ramified q series}
Let~$p$ be a prime number and $k$~a field of characteristic~$p$. 
Moreover, let~$\gamma$ be a root of unity in~$k$, denote by~$q$ its order, 
and let~$g(\zeta)$ be a power series in~$k[[\zeta]]$ of the form~\eqref{e:reduced form reloaded}.
If~$p$ is odd (resp.~$p = 2$), then~$g$ is minimally ramified if and only if
$$ a_1 \neq 0
\text{ and }
a_2 \neq \frac{q + 1}{2} a_1^2
\quad
\left( \text{resp. }
a_1 \neq 0,
a_2 \neq 0,
\text{ and }
a_2 \neq a_1^2 \right). $$
\end{coro}

\begin{proof}
When~$p$ is odd the corollary is a direct consequence of~\eqref{e:generic form g after q iterates} and~\eqref{e:generic form g iterates} 
in the Main Lemma.
In the case $p=2$ the integer~$q$ is odd, so we have~$q = 1$ in~$k$, and therefore by~\eqref{e:generic form g after q iterates} we have~$g^q(\zeta) \equiv \zeta (1 + a_1 \zeta^q) \mod \langle \zeta^{2q + 1} \rangle$.
Moreover, in~$k$ we have either
$$ \frac{q-1}{2} = 0
\text{ and }
\frac{q+1}{2} = 1,
\text{ or }
\frac{q-1}{2} = 1
\text{ and }
\frac{q+1}{2} = 0, $$
so either~$\chi_{q, 1} = a_1(a_1^2 - a_2)$ or~$\chi_{q, 1} = -a_1a_2$, and for each integer~$n \ge 2$ we have either
$$ \chi_{q, n} = - a_1 a_2^{2^{n - 1} - 1}(a_1^2 - a_2)^{2^{n - 1}}
\text{ or }
\chi_{q, n} = a_1a_2^{2^{n - 1}}(a_1^2 - a_2)^{2^{n - 1} - 1}. $$
In view of~\eqref{e:generic form g iterates}, we conclude that~$g$ is minimally ramified if and only if~$a_1 \neq 0$, $a_2 \neq 0$, and~$a_2 \neq a_1^2$.
This completes the proof of the corollary.
\end{proof}

\begin{coro}
\label{c:generic periodic points lower bound g q}
Let~$p$ be a prime number and $(k, | \cdot |)$ an ultrametric field of
characteristic~$p$.
Moreover, let~$\gamma$ be a root of unity in~$k$, denote by~$q$ its order, and let~$g(\zeta)$ be a power series
in~$\Ok[[\zeta]]$ of the form~\eqref{e:reduced form reloaded}.
Then for every integer~$n\geq 1$ and every periodic point~$\zeta_0$ 
in $\mathfrak{m}_k$ 
 of~$g$ of minimal period~$qp^n$, we have
\begin{equation}
  \label{e:generic periodic points lower bound g q}
|\zeta_0|
\ge
\begin{cases}
\left |a_1 \right|^{\frac{1}{q}-\frac{2}{qp}}
\cdot \left|\frac{q+1}{2}a_1^2-a_2 \right|^{\frac{1}{qp}}
& \text{if~$p$ is odd;}
\\
\left|\frac{q+1}{2}a_1^2-a_2 \right|^{\frac{1}{2q}}
& \text{if~$p = 2$ and~$n = 1$;}
\\
\left|a_2\left(a_1^2-a_2\right)\right|^{\frac{1}{4q}}
& \text{if~$p = 2$ and~$n \ge 2$.}
\end{cases}
\end{equation}
Moreover, if~$w_0$ is a periodic point of~$g$ in~$\mk \setminus \{ 0
\}$ of minimal period~$q$, then~$|w_0|\geq |a_1|^{\frac{1}{q}}$.
If in addition~$a_1=0$, then we 
also
have the bound~$|w_0|\geq |a_2|^{\frac{1}{q}}$.
\end{coro}

\begin{proof}
We use several times the fact that by definition~$p$ does not divide~$q$ and so~$|q|=1$. 

We first consider the case~$n=0$.
If~$a_1\neq 0$, then by~\eqref{e:generic form g after q iterates} in the Main Lemma we have~$\delta_0(g^q) = qa_1$. 
Together with~\eqref{e:fixed bound} in Lemma~\ref{l:periodic bound},
this implies~$|w_0|\geq |a_1|^{\frac{1}{q}}$.
Suppose~$a_1 = 0$.
Then we have~$|w_0|\geq |a_1|^{\frac{1}{q}}$ trivially.
If~$a_2 = 0$, then the inequality~$|w_0| \ge |a_2|^{\frac{1}{q}}$ also
holds trivially.
If~$a_2\neq 0$, then by~\eqref{e:generic form g after q iterates} in the Main Lemma we have~$\delta_0(g^q)=qa_2$. 
Then by~\eqref{e:fixed bound} in Lemma~\ref{l:periodic bound} we obtain the non-trivial bound~$|w_0|\geq |a_2|^{\frac{1}{q}}$.

We now consider the case where~$n\geq 1$ and~$\chi_{q, n}\neq 0$.
By the Main Lemma we have~$\delta_n(g^q) = \chi_{q, n}$.
If in addition~$n \ge 2$, then our assumption~$\chi_{q, n} \neq 0$
implies~$\chi_{q, n - 1} \neq 0$, and by the Main Lemma we have~$\delta_{n - 1}(g^q) = \chi_{q, n - 1}$.
Combined with~\eqref{e:periodic bound} in Lemma~\ref{l:periodic
  bound}, this implies
\begin{equation*}
  |\zeta_0|
\geq
\left|\frac{\chi_{q, n}}{\chi_{q,n-1}}\right|^{\frac{1}{qp^n}}
=
\begin{cases}
\left |a_1 \right|^{\frac{1}{q}-\frac{2}{qp}}
\cdot \left|\frac{q+1}{2}a_1^2-a_2 \right|^{\frac{1}{qp}}
& \text{if~$p$ is odd;}
\\
\left|a_2\left(a_1^2-a_2\right)\right|^{\frac{1}{4q}}
& \text{if~$p = 2$.}
\end{cases}
\end{equation*}
Suppose~$n = 1$.
Our assumption~$\chi_{q, 1} \neq 0$ implies~$a_1 \neq 0$, and by~\eqref{e:generic form g after q iterates} in the Main
Lemma we have~$\delta_0(g^q) = qa_1$.
Then by~\eqref{e:periodic bound} in Lemma~\ref{l:periodic bound}, we have
\begin{equation*}
|\zeta_0|
\ge
\left | \frac{\chi_{q, 1}}{q a_1}\right|^{\frac{1}{qp}}
=
\begin{cases}
\left |a_1 \right|^{\frac{1}{q}-\frac{2}{qp}}
\cdot \left|\frac{q+1}{2}a_1^2-a_2 \right|^{\frac{1}{qp}}
& \text{if~$p$ is odd;}
\\
\left|\frac{q+1}{2}a_1^2-a_2 \right|^{\frac{1}{2q}}
& \text{if~$p = 2$.}
\end{cases}
\end{equation*}
 
It remains to consider the case where~$n\geq 1$ and~$\chi_{q, n}=0$.
If~$p$ is odd, then~$\chi_{q, n} = 0$ implies~$a_1
= 0$ or~$\frac{q+1}{2} a_1^2 - a_2 = 0$.
In both cases the inequality~\eqref{e:generic periodic points lower bound g q} holds trivially.
Suppose~$p=2$ and~$n = 1$.
Then our assumption~$\chi_{q, 1} = 0$ implies~$a_1 = 0$ or~$\frac{q +
  1}{2} a_1^2 - a_2 = 0$.
In the latter case the inequality~\eqref{e:generic periodic points lower bound g q} holds
trivially, so we assume~$a_1 = 0$ and~$\frac{q + 1}{2} a_1^2 - a_2
\neq 0$, and therefore~$a_2 \neq 0$.
Then 
by~\ref{e:generic form g after q iterates} and~\eqref{e:generic form g
  iterates} in
the Main Lemma we
respectively
 have
$$ \delta_0(g^q) = q a_2
\text{ and }
\delta_1(g^q) = \xi_{q, 1} = - a_2^2. $$
So, by~\eqref{e:periodic bound} in Lemma~\ref{l:periodic
  bound} we have~$|\zeta_0| \ge \left| a_2 \right|^{\frac{1}{2q}}$,
which is~\eqref{e:generic periodic points lower bound g q} in this particular case.
It remains to consider the case where~$p = 2$ and~$n \ge 2$.
Note that our assumption~$\chi_{q, n} = 0$ implies that~$\chi_{q, n -
  1} = 0$.
If either~$a_1 = 0$ or~$a_2 = a_1^2$, then the inequality~\eqref{e:generic periodic
  points lower bound g q} holds trivially, so we assume~$a_1 \neq 0$
and~$a_2 \neq a_1^2$.
Then~$\xi_{q, n} \neq 0$ and~$\xi_{q, n - 1} \neq 0$.
Since~$\chi_{q, n} = \chi_{q, n - 1} = 0$, by the Main Lemma we have
$$ \delta_n(g^q)=\xi_{q, n}
\text{ and }
\delta_{n - 1}(g^q) = \xi_{q, n - 1}. $$
Together with~\eqref{e:periodic bound} in Lemma~\ref{l:periodic
  bound}, this implies
\begin{equation*}
|\zeta_0|
\geq
\left | \frac{\xi_{q, n}}{\xi_{q, n - 1}} \right|^{\frac{1}{2^n q}}
=
\left| a_2 \left(a_1^2 - a_2 \right) \right|^{\frac{1}{4q}}.
\end{equation*}
This completes the proof of the corollary.
\end{proof}

\subsection{Proof of Theorems~\ref{t:genericity}, \ref{t:isolated}, and~\ref{t:characterization of minimally ramified} assuming the Main Lemma}
\label{ss:proof of genericity and isolated}
The proofs are given after the following lemma, which is an enhanced version of~\cite[Proposition~$4.1$]{LinRiv16}.

\begin{lemm}
\label{l:basic normal form reloaded}
Let~$p$ be a prime number, let~$k$ a field of characteristic~$p$, and
denote by~$F$ the prime field of~$k$.
Fix a root of unity~$\gamma$ in~$k$ different from~$1$, and denote
by~$q \ge 2$ its order.
Then for every integer~$\ell \ge 1$ there are polynomials
$$ \alpha_\ell(x_1, \ldots, x_{\ell})
\text{ and }
\beta_\ell(x_1, \ldots, x_{\ell q}) $$
with coefficients in~$F(\gamma)$, such that the following property
holds.
For every power series~$f(\zeta)$ in~$k[[\zeta]]$ of the form
$$ f(\zeta)
=
\gamma \zeta \left( 1 + \sum_{i = 1}^{+ \infty} c_i \zeta^i \right), $$
the power series
$$ h(\zeta)
\=
\zeta \left( 1 + \sum_{\ell = 1}^{+ \infty} \alpha_{\ell} (c_1,
  \ldots, c_{\ell}) \zeta^{\ell} \right) $$
in~$k[[\zeta]]$ is such that
$$ h \circ f \circ h^{-1}(\zeta)
=
\gamma \zeta \left( 1 + \sum_{j = 1}^{+ \infty} \beta_j (c_1,
  \ldots, c_{j q}) \zeta^{j q} \right). $$
\end{lemm}

\begin{proof}
Put~$F_{\infty} \= F(\gamma)[x_1, x_2, \ldots]$ and for each integer~$m \ge 1$
put~$F_m \= F(\gamma)[x_1, \ldots, x_m]$.
Moreover, consider the power series~$\whf(\zeta)$ in~$F_{\infty}[[\zeta]]$
defined by
$$ \whf(\zeta)
\=
\gamma \zeta \left( 1 + \sum_{i = 1}^{+ \infty} x_i \zeta^i \right). $$

Let~$s_0(\zeta)$ and~$h_0(\zeta)$ be the polynomials in~$F(\gamma)[\zeta]$ defined
by
$$ s_0(\zeta) \= 1
\text{ and }
h_0(\zeta) \= \zeta. $$
We define inductively for every integer~$\ell \ge 1$
polynomials~$s_{\ell}(\zeta)$ and~$h_{\ell}(\zeta)$ in~$F_\ell[\zeta]$ of degrees at
most~$\ell + 1$ and~$\left[ \frac{\ell}{q} \right]$, respectively, such that
  \begin{align}
\label{e:finite level conjugacy A}
h_{\ell}(\zeta)
& \equiv
h_{\ell - 1}(\zeta) \mod \left\langle \zeta^{\ell + 1} \right\rangle
\\ \intertext{and}
\label{e:finite level conjugacy B}
s_{\ell}(\zeta)
& \equiv
s_{\ell - 1}(\zeta) \mod \left\langle \zeta^{\left[ \frac{\ell - 1}{q} \right]} \right\rangle,
\\ \intertext{and such that the power series~$\whf_{\ell}(\zeta) \= h_{\ell} \circ \whf \circ
h_{\ell}^{-1}(\zeta)$ in~$F_{\infty}[[\zeta]]$ satisfies}
\label{e:finite level clearing}
\whf_{\ell}(\zeta)
& \equiv
\gamma \zeta s_{\ell}(\zeta^q) \mod \left\langle \zeta^{\ell + 2} \right\rangle.  
  \end{align}
Note that
$$ \whf_0(\zeta)
\=
h_0 \circ \whf \circ h_0^{-1}(\zeta)
=
\whf(\zeta)
\equiv
\gamma \zeta \mod \left\langle \zeta^2 \right\rangle, $$
so~\eqref{e:finite level clearing} is satisfied when~$\ell = 0$.
Let~$\ell \ge 1$ be an integer for which~$s_{\ell - 1}(\zeta)$
and~$h_{\ell - 1}(\zeta)$ are already defined and
satisfy~\eqref{e:finite level clearing} with~$\ell$ replaced by~$\ell
- 1$.
Then there is~$A(x_1, \ldots, x_\ell)$ in~$F_\ell$ such that
$$ \whf_{\ell - 1}(\zeta)
\equiv
\gamma \zeta \left( s_{\ell - 1}(\zeta^q) + A(x_1, \ldots, x_\ell) \zeta^\ell \right) \mod
\left\langle \zeta^{\ell + 2} \right\rangle. $$
In the case where~$\ell$ is divisible by~$q$, the congruence~\eqref{e:finite level
  clearing} is verified if we put
$$ h_{\ell}(\zeta) \= h_{\ell - 1}(\zeta)
\text{ and }
s_{\ell}(\zeta) \= s_{\ell - 1}(\zeta) + A(x_1, \ldots, x_\ell) \zeta^{\frac{\ell}{q}}. $$
Suppose~$\ell$ is not divisible by~$q$.
Then~$\gamma^\ell - 1 \neq 0$, and
$$ B(x_1, \ldots, x_\ell)
\= -
(\gamma^\ell - 1)^{-1} A(x_1, \ldots, x_\ell) $$
defines a polynomial in~$F_\ell$.
Consider the polynomial
$$ h(\zeta) \= \zeta \left( 1 + B(x_1, \ldots, x_\ell) \zeta^\ell \right) $$
in~$F_\ell[\zeta]$ and put
$$ h_{\ell}(\zeta) \= h \circ h_{\ell - 1}(\zeta)
\text{ and }
s_{\ell}(\zeta) \= s_{\ell - 1}(\zeta). $$
Then we have
$$ \whf_{\ell}(\zeta)
=
h_\ell \circ \whf \circ h_{\ell}^{-1}(\zeta)
=
h \circ \whf_{\ell - 1} \circ h^{-1}(\zeta), $$
so there is~$C(x_1, \ldots, x_\ell)$ in~$F_{\ell}$ such that
\begin{equation*}
  \begin{split}
\whf_{\ell}(\zeta)
& \equiv
\gamma \zeta \left( s_{\ell - 1}(\zeta^q) + C(x_1, \ldots, x_\ell) \zeta^\ell \right)
\mod \left\langle \zeta^{\ell + 2} \right\rangle
\\ & \equiv
\gamma \zeta \left( s_{\ell}(\zeta^q) + C(x_1, \ldots, x_\ell) \zeta^\ell \right)
\mod \left\langle \zeta^{\ell + 2} \right\rangle.
  \end{split}
\end{equation*}
Thus, to complete the proof of the induction step it is enough to show that~$C(x_1, \ldots, x_\ell) = 0$.
To do this, note that by our definition of~$B(x_1, \ldots, x_\ell)$ we have
\begin{align*}
h \circ \whf_{\ell - 1}(\zeta)
& =
\whf_{\ell - 1}(\zeta) \left( 1 + B(x_1, \ldots, x_\ell) \whf_{\ell - 1}(\zeta)^\ell \right)
\\ & \equiv
\gamma \zeta \left( s_{\ell - 1}(\zeta^q) + A(x_1, \ldots, x_\ell) \zeta^\ell \right)
\left( 1 + B(x_1, \ldots, x_\ell) \gamma^{\ell} \zeta^{\ell} \right)
\\ & \qquad
\mod \left\langle \zeta^{\ell + 2} \right\rangle
\\ & \equiv
\gamma \zeta \left( s_{\ell - 1}(\zeta^q) + \left( A(x_1, \ldots, x_\ell) + B(x_1, \ldots, x_\ell) \gamma^{\ell} \right) \zeta^{\ell} \right)
\\ & \qquad
\mod \left\langle \zeta^{\ell + 2} \right\rangle
\\ &
\equiv
\gamma \zeta \left( s_{\ell - 1}(\zeta^q) + B(x_1, \ldots, x_\ell) \zeta^{\ell} \right)
\mod \left\langle \zeta^{\ell + 2} \right\rangle.
\end{align*}
On the other hand
\begin{equation*}
  \begin{split}
\whf_{\ell} \circ h(\zeta)
& \equiv
\gamma h(\zeta) \left( s_{\ell - 1}(h(\zeta)^q) + C(x_1, \ldots, x_\ell) h(\zeta)^\ell \right)
\mod \left\langle \zeta^{\ell + 2} \right\rangle
\\ & \equiv
\gamma \zeta \left( 1 + B(x_1, \ldots, x_\ell) \zeta^{\ell} \right)
\left( s_{\ell - 1} \left( \zeta^q \right) + C(x_1, \ldots, x_\ell) \zeta^{\ell} \right)
\\ & \qquad
\mod \left\langle \zeta^{\ell + 2} \right\rangle
\\ & \equiv
\gamma \zeta \left( s_{\ell - 1} \left( \zeta^q \right) + (B(x_1, \ldots, x_\ell) + C(x_1, \ldots, x_\ell)) \zeta^{\ell} \right)
\\ & \qquad
\mod \left\langle \zeta^{\ell + 2} \right\rangle.
  \end{split}
\end{equation*}
Comparing coefficients we conclude that~$B(x_1, \ldots, x_\ell) = 0$.
This completes the proof of the induction step and of~\eqref{e:finite level clearing}.

For each integer~$\ell \ge 1$ let~$\alpha_{\ell}(x_1, \ldots, x_{\ell})$ be the coefficient
of~$\zeta^{\ell + 1}$ in~$h_\ell(\zeta)$ and let~$\beta_\ell(x_1, \ldots, x_{q\ell})$ be the
coefficient of~$\zeta^{\ell}$ in~$s_{q\ell}(\zeta)$.
Then the power series
$$ \whh(\zeta)
\=
\zeta \left( 1 + \sum_{\ell = 1}^{+ \infty} \alpha_\ell(x_1, \ldots, x_\ell)
  \zeta^\ell \right) $$
in~$F_{\infty}[[\zeta]]$ is invertible, and by~\eqref{e:finite level
  conjugacy A}, \eqref{e:finite level conjugacy B}, and~\eqref{e:finite level clearing} we have
$$ \whh \circ \whf \circ \whh^{-1}(\zeta)
=
\gamma \zeta \left( 1 + \sum_{j = 1}^{+ \infty} \beta_j (x_1, x_1, \ldots, x_{j q}) \zeta^{j q} \right). $$
The proposition is obtained by specializing, for each integer~$i \ge
1$, the variable~$x_i$ to~$c_i$.
\end{proof}

\begin{proof}[Proof of Theorem~\ref{t:genericity}]
Put
$$ M_1(x_1,x_2)
\=
\begin{cases}
  x_1 \left( x_1^2 - x_2 \right)
& \text{if~$p$ is odd;}
\\
x_1 x_2 \left( x_1^2 - x_2 \right)
& \text{if~$p = 2$,}
\end{cases} $$ 
and for~$q \ge 2$, let~$\beta_1(x_1, \ldots, x_q)$ and~$\beta_2(x_1, \ldots, x_{2q})$ be the
polynomials given by Lemma~\ref{l:basic normal form reloaded}, and put
$$ M_q(x_1, \ldots, x_{2q})
\=
\beta_1(x_1, \ldots, x_q) \left( \frac{q + 1}{2} \beta_1(x_1, \ldots, x_q)^2 - \beta_2(x_1, \ldots, x_{2q})
\right), $$
if~$p$ is odd, and
\begin{multline*}
  M_q(x_1, \ldots, x_{2q})
\\ \=
\beta_1(x_1, \ldots, x_q) \beta_2(x_1, \ldots, x_{2q}) \left( \beta_1(x_1, \ldots, x_q)^2 - \beta_2(x_1, \ldots, x_{2q})
\right),
\end{multline*}
if~$p = 2$.

By Corollary~\ref{c:minimally ramified q series}, in all the cases we have that a power series~$f(\zeta)$ in~$k[[\zeta]]$ of the form~\eqref{e:typical series}
is minimally ramified if and only if~$M_q(c_1, \ldots, c_{2q}) \neq 0$.
So, it only remains to prove 
that~$M_q(x_1,\ldots, x_{2q})\not\equiv 0$. 
When~$q = 1$ this follows from the definition.
To prove that~$M_q(x_1,\ldots, x_{2q})$ is nonzero when~$q \ge 2$, note
that by Corollary~\ref{c:minimally ramified q series} at least~$1$ of
the following polynomials in~$F(\gamma)[\zeta]$ is minimally ramified:
$$ \gamma \zeta \left( 1 + \zeta^q \right),
\gamma \zeta \left( 1 + \zeta^q + \gamma \zeta^{2q} \right). $$
Thus, either
$$ M_q (\underbrace{0, \ldots, 0}_{q - 1}, 1, \underbrace{0, \dots,
  0}_{q})
\text{ or }
M_q (\underbrace{0, \ldots, 0}_{q - 1}, 1, \underbrace{0, \ldots,
  0}_{q - 1}, \gamma) $$
is nonzero.
This completes the proof of the theorem.
\end{proof}

\begin{proof}[Proof of Theorem~\ref{t:isolated}]
Suppose~$\gamma = 1$, so that~$q = 1$, and let~$a_1$ and~$a_2$ be such that
$$ f(\zeta)
\equiv
\zeta \left( 1 + a_1 \zeta + a_2 \zeta^2 \right)
\mod \left\langle \zeta^4 \right\rangle. $$
Since~$i_0(f) = 1$, we have~$a_1 \neq 0$ and~$\delta_0(f)
= a_1$, so by definition
$$ \resit(f) = 1 - \frac{a_2}{a_1^2}. $$
So the desired assertion is given by Corollary~\ref{c:generic periodic points lower bound g q}.

Suppose~$\gamma \neq 1$, so that~$q \ge 2$.
By Lemma~\ref{l:basic normal form reloaded} there is a power
series~$h(\zeta)$ in~$\Ok[[\zeta]]$ of the form
$$ h(\zeta)
=
\zeta \left( 1 + \sum_{\ell = 1}^{+ \infty} \alpha_{\ell} \zeta^{\ell} \right), $$
such that~$g(\zeta) \= h \circ f \circ h^{-1}(\zeta)$ is of
the form~\eqref{e:reduced form reloaded}.
In particular, $h(\zeta_0)$ is a periodic point of~$g$ in~$\mk$ of minimal
period~$qp^n$.
Furthermore,
$$ \left| h(\zeta_0) - \zeta_0 \right|
=
|\zeta_0| \cdot \left| \sum_{\ell = 1}^{+ \infty} \alpha_{\ell}
  \zeta_0^{\ell} \right|
\le
|\zeta_0|^2
<
|\zeta_0|, $$
so~$|h(\zeta_0)| = |\zeta_0|$.
On the other hand, by Lemma~\ref{l:quasi-invariants}, 
\eqref{e:generic form g after q iterates} in the Main Lemma, 
the fact that $f'(0) = g'(0)$,
and our
hypothesis~$i_0(f^q) = q$, we have
\begin{equation}
  \label{e:normalized invariants}
i_0(f^q) = i_0(g^q) = q
\text{ and }
\delta_0(f^q) = \delta_0(g^q) = qa_1.
\end{equation}
So, $a_1 \neq 0$ and by definition
\begin{equation}
\label{e:resit}
  \resit(f) = \frac{q + 1}{2} - \frac{a_2}{a_1^2}.
\end{equation}
Then the desired estimate is a direct consequence of
Corollary~\ref{c:generic periodic points lower bound g q} applied to~$g$, 
using~$|h(\zeta_0)| = |\zeta_0|$
and that~$q$ is not divisible so~$|q| = 1$.  
\end{proof}

\begin{proof}[Proof of Theorem~\ref{t:characterization of minimally ramified}]
If~$\gamma = 1$, and therefore~$q = 1$, then the desired assertion is given
by Corollary~\ref{c:minimally ramified q series} with~$g = f$.

Suppose~$\gamma \neq 1$, so that~$q \ge 2$.
By Lemma~\ref{l:basic normal form reloaded} there is a power
series~$h(\zeta)$ in~$k[[\zeta]]$ satisfying~$h(\zeta) \equiv \zeta
\mod \left\langle \zeta^2 \right\rangle$, such that~$g(\zeta) \= h
\circ f \circ h^{-1}(\zeta)$ is of the form~\eqref{e:reduced form reloaded}.
Suppose~$f$ is minimally ramified.
Then~$i_0(f^q) = q$ and by Lemma~\ref{l:quasi-invariants} the power
series~$g$ is minimally ramified.
Moreover, by Corollary~\ref{c:minimally ramified q series} we have~$a_1 \neq 0$, so by definition we have~\eqref{e:resit}.
Then Corollary~\ref{c:minimally ramified q series} implies
that~$\resit(f) \neq 0$, and when~$p = 2$ that~$\resit(f) \neq 1$.
This completes the proof of the direct implication when~$q \ge 2$.
To prove the reverse implication, suppose~$i_0(f^q) = q$, $\resit(f)
\neq 0$, and when~$p = 2$ that~$\resit(f) \neq 1$.
By Lemma~\ref{l:quasi-invariants} and~\eqref{e:generic form g after q
  iterates} in the Main Lemma, we have~\eqref{e:normalized invariants}
and~$a_1 \neq 0$.
So, by definition we have~\eqref{e:resit}.
Thus our assumptions on~$\resit(f)$ and Corollary~\ref{c:minimally
  ramified q series} imply that~$g$ is minimally ramified.
By Lemma~\ref{l:quasi-invariants} the power series~$f$ is also
minimally ramified.
This completes the proof of the reverse implication in the case~$q \ge
2$, and of the theorem.
\end{proof}

\section{Generic parabolic points}
\label{s:proof of the main lemma}
The purpose of this section is to prove the Main Lemma stated in~\S\ref{s:a reduction}.
The main ingredient is the following lemma, which is the case where~$q = 1$ 
and $n\geq 1$.
Note that the Main Lemma is trivial in the case where $q=1$ and $n=0$.
The proof of the Main Lemma is at the end of this section.

\begin{lemm}[The case $q=1$ and $n\geq 1$]
\label{l:generic form q=1}
Let~$p$ be a prime number, $k$ an ultrametric field of characteristic~$p$,
and
$$ f(\zeta) = \zeta  + a\zeta^2 + b\zeta^3 + \cdots $$
a power series in~$k[[\zeta]]$.
Given an integer~$n \ge 1$, put
\begin{align*}
\chi_n
& \=
\begin{cases}
a^{p^n-\frac{p^n-1}{p-1}}\left(a^2 - b \right)^{\frac{p^{n}-1}{p-1}}
& \text{if~$p$ is odd;}
\\
a\left(a^2 - b \right)
& \text{if~$p = 2$ and~$n = 1$;}
\\
ab^{2^{n-1}-1}\left(a^2 - b \right)^{2^{n-1}}
& \text{if~$p = 2$ and~$n \ge 2$,}
\end{cases}
\\ \intertext{and}
\xi_n
& =
\begin{cases}
-a^{p^n-\frac{p^n-1}{p-1}-1}\left(a^2 - b \right)^{\frac{p^{n}-1}{p-1}+1}
& \text{if~$p$ is odd;}
\\
b^{2^{n-1}}\left(a^2 - b \right)^{2^{n-1}}
& \text{if~$p = 2$.}
\end{cases}
\end{align*}
Then we have
\begin{equation}
  \label{e:generic form}
f^{p^n}(\zeta)
\equiv
\zeta \left( 1 + \chi_n\zeta^{\frac{p^{n+1}-1}{p-1}} +
  \xi_n\zeta^{\frac{p^{n+1}-1}{p-1}+1} \right)
\mod \left\langle \zeta^{\frac{p^{n+1}-1}{p-1} + 3} \right\rangle.
\end{equation}
\end{lemm}

\begin{proof}
The proof is divided into Cases 1.1, 1.2, 1.3 and 2.

\partn{Case 1.1}~$p=3$ and~$n=1$.
For each integer~$m\geq 1$ define the power series~$\Delta_m(\zeta)$ in~$k[[\zeta]]$
inductively by~$\Delta_1(\zeta)\=f(\zeta)-\zeta$, 
and for~$m\geq 2$ by
$$\Delta_m(\zeta)\=\Delta_{m-1}(f(\zeta))-\Delta_{m-1}(\zeta).$$
Note that~$\Delta_3(\zeta)=f^3(\zeta)-\zeta$.
By definition~$\Delta_1(\zeta)$ satisfies
$$
\Delta_1(\zeta)
\equiv a\zeta^2+b\zeta^3+c\zeta^4
\mod\left\langle\zeta^{5}\right\rangle.
$$
Using~$p=3$, we have
\begin{equation*}
 \begin{split}
\Delta_{2}(\zeta)
& \equiv
a\zeta^2\left[\left(1+a\zeta+b\zeta^{2}+c\zeta^3\right)^{2}-1\right]
+c\zeta^4\left[(1+a\zeta)^{4}-1\right]
\mod\left\langle\zeta^{6}\right\rangle.
  \end{split}
\end{equation*}
Note that 
\begin{equation*}
\begin{split}
\left( 1+a\zeta+b\zeta^{2} +c\zeta^3 \right)^{2}-1
& \equiv
2a\zeta+2b\zeta^2+2c\zeta^3 +a^2\zeta^2+2ab\zeta^3
\mod\left\langle\zeta^{4}\right\rangle
\\ & \equiv
- a\zeta
+ \left( a^2-b \right)\zeta^2-(ab+c)\zeta^3
\mod\left\langle\zeta^{4}\right\rangle.
\end{split}
\end{equation*}
Consequently, using~$p=3$ we thus have
\begin{equation*}
 \begin{split}
   \Delta_2(\zeta)
& \equiv
-a^2\zeta^ 3 + a \left( a^2 - b \right) \zeta^4 - \left( a^2b+ac \right) \zeta^5+ac\zeta^5
\mod\left\langle\zeta^{6}\right\rangle
\\ & \equiv
-a^2\zeta^3 + a \left( a^2 - b \right)\zeta^4-a^2b\zeta^5
\mod\left\langle\zeta^{6}\right\rangle,
  \end{split}
\end{equation*}
and
\begin{equation*}
 \begin{split}
   \Delta_3(\zeta)
& \equiv
-a^2 \zeta^3\left[ \left( 1+a\zeta+b\zeta^{2} \right)^{3}-1\right]
   +a \left( a^2-b \right) \zeta^4\left[ \left( 1+a\zeta+b\zeta^{2} \right)^{4}-1\right]
\\ & \quad
- a^2b\zeta^5\left[(1+a\zeta)^{5}-1\right]
\mod\left\langle\zeta^{7}\right\rangle
   \\ & \equiv
- a^5\zeta^6+a^2 \left( a^2-b \right) \zeta^5
+ ab \left( a^2-b \right) \zeta^6+a^3b\zeta^6
\mod\left\langle\zeta^{7}\right\rangle
    \\ & \equiv
a^2 \left( a^2-b \right) \zeta^5-a \left( a^2-b \right)^2 \zeta^6
\mod\left\langle\zeta^7\right\rangle.
  \end{split}
\end{equation*}
This completes the proof in the case where~$p=3$ and~$n=1$.

\partn{Case 1.2}~$p\geq 5$ and~$n=1$.
For each integer~$m\geq 1$ define the power series~$\Delta_m(\zeta)$ in~$k[[\zeta]]$
inductively by~$\Delta_1(\zeta)\=f(\zeta)-\zeta$, 
and for~$m\geq 2$ by
$$\Delta_m(\zeta)\=\Delta_{m-1}(f(\zeta))-\Delta_{m-1}(\zeta).$$
Note that~$\Delta_p(\zeta)=f^p(\zeta)-\zeta$.

We first prove
\begin{equation}
\label{e:delta_p-3}
\Delta_{p-3}(\zeta)\equiv-\frac{1}{2}a^{p-3}\zeta^{p-2}-\frac{3}{2}a^{p-2}\zeta^{p-1}\mod\left\langle\zeta^p\right\rangle.
\end{equation}
To prove this, 
for integers~$m\geq 1$ define~$\alpha_m$ and~$\beta_m$ in~$k$ inductively by 
\begin{equation}
\label{e:alpha recurrence}
\alpha_1\=0, \quad \alpha_{m+1}\=(m+2)\alpha_m+(m+1)!\frac{m}{2},
\end{equation}
and
\begin{equation}
\label{e:beta recurrence}
\beta_1\=1, \quad \beta_{m+1}\=(m+2)\beta_{m}+(m+1)!.
\end{equation}
We prove by induction that for~$m\geq 1$ we have
\begin{equation}
\label{e:delta_m}
 \Delta_m(\zeta)\equiv m!a^m\zeta^{m+1}
+ \left( \alpha_ma^{m+1}+\beta_ma^{m-1}b \right) \zeta^{m+2}
\mod \left\langle\zeta^{m+3}\right\rangle.
\end{equation}
When~$m=1$ this is true by definition. 
Let~$m\geq 1$ be such that~\eqref{e:delta_m} holds.
Then
\begin{multline*}
 \begin{aligned}
  \Delta_{m+1}(\zeta) 
& \equiv
   m!a^{m}\zeta^{m+1}\left[ \left( 1+a\zeta+b\zeta^2 \right)^{m+1} -1\right]
\\ & \quad
  + \left( \alpha_ma^{m+1}+\beta_ma^{m-1}b \right) \zeta^{m+2}
\left[ (1+ a\zeta)^{m+2}-1 \right]
 \end{aligned}
\\
\mod \left\langle\zeta^{m+4}\right\rangle.
\end{multline*}
Note that
\begin{multline*}
\left( 1+a\zeta + b \zeta^2 \right)^{m+1} -1
\equiv   
(m+1)a\zeta+\left((m+1)b+\frac{(m+1)m}{2}a^2\right)\zeta^2
\\
\mod \left\langle \zeta^3\right\rangle. 
\end{multline*}
Consequently,
\begin{equation*}
 \begin{split}
   \Delta_{m+1}(\zeta) 
& \equiv
   (m+1)!a^{m+1}\zeta^{m+2}+\left( (m+1)!a^mb+(m+1)!\frac{m}{2}a^{m+2}\right)\zeta^{m+3}
\\ & \quad
  +\left( (m+2)\alpha_ma^{m+2}+(m+2)\beta_ma^{m}b\right )\zeta^{m+3} \mod \left\langle\zeta^{m+4}\right\rangle.
 \end{split}
\end{equation*}
This completes the proof of the induction step and proves~\eqref{e:delta_m}.
To prove~\eqref{e:delta_p-3},
first note that by~\eqref{e:alpha recurrence} and using~$\alpha_1=0$,
we have for~$m$ in~$\{2, \ldots, p - 3 \}$
\begin{equation*}
 \begin{split}
  \frac{\alpha_m}{(m+1)!}
  & =
  \frac{\alpha_{m-1}}{m!}+\frac{1}{2} \cdot \frac{m-1}{m+1}
  \\ & =
\frac{\alpha_{m-1}}{m!}+\frac{1}{2}-\frac{1}{m+1}
  \\& =
\frac{\alpha_{1}}{2!}+\frac{m-1}{2} - \left(
  \frac{1}{3}+\dots+\frac{1}{m+1} \right)
  \\& =
\frac{m-1}{2} - \left( \frac{1}{3}+\dots+\frac{1}{m+1} \right)
  \\& =
\frac{m+2}{2} - \left( 1+\frac{1}{2}+\frac{1}{3}+\dots+\frac{1}{m+1} \right).
 \end{split}
\end{equation*}
Consequently, for~$m = p - 3$, using~$(p-2)!=1$ we obtain~$\alpha_{p -
  3} = - \frac{3}{2}$.
So, to complete the proof of~\eqref{e:delta_p-3} it is enough to show
that~$\beta_{p - 3} = 0$.
To do this, note that using~$\beta_1=1$ we have for every~$m$ in~$\{2,
\ldots, p - 3 \}$
\begin{equation*}
 \begin{split}
  \frac{\beta_m}{(m+1)!}
  & =
  \frac{\beta_{m-1}}{m!}+\frac{1}{m+1}
  \\& = \frac{\beta_{1}}{2!}+\frac{1}{3}+\dots+\frac{1}{m+1}
  \\& = \frac{1}{2}+\frac{1}{3}+\dots+\frac{1}{m+1}.
 \end{split}
\end{equation*}
Hence,  for~$m=p-3$ we obtain~$\beta_{p - 3} = 0$, as required.
This completes the proof of~\eqref{e:delta_p-3}.

To complete the proof of~\eqref{e:generic form} for~$p\geq 5$
and~$n=1$, define~$A$, $B$, and~$C$ in~$k$ by
$$
\Delta_{p-3}(\zeta)=A\zeta^{p-2}+B\zeta^{p-1}+C\zeta^{p}+\cdots .
$$
Note that by~\eqref{e:delta_p-3} we have 
\begin{equation}
\label{e:AB}
A=-\frac{1}{2}a^{p-3}, \text{ and } B=-\frac{3}{2}a^{p-2}=3Aa.
\end{equation}
By~\eqref{e:delta_p-3} and by definition of~$\Delta_{p - 2}$, we have
\begin{equation*}
 \begin{split}
  \Delta_{p-2}(\zeta)
  & \equiv
  A\zeta^{p-2}\left[ \left( 1+a\zeta +b\zeta^2 +c\zeta^3 \right)^{p-2}-1\right] 
  \\ & \quad
+  3Aa\zeta^{p-1}\left[ \left( 1+a\zeta+b\zeta^2 \right)^{p-1}-1\right]
  \\ & \quad
+C\zeta^p\left[(1+a\zeta)^p-1\right] 
  \mod\left\langle \zeta^{p+2}\right\rangle.
 \end{split}
\end{equation*}
Note that 
\begin{multline*}
\left( 1+a\zeta 
  +
  b\zeta^2 +c\zeta^3 \right)^{p-2}-1
\equiv
-2a\zeta + \left( 3a^2-2b \right)\zeta^2
+ \left( -4a^3+6ab-2c \right)\zeta^3
\\
\mod \left\langle\zeta^4\right\rangle,
\end{multline*}
and
\begin{equation*}
\left( 1+a\zeta + b\zeta^2 \right)^{p-1}-1
\equiv
-a\zeta + \left( a^2-b \right)\zeta^2 \mod \left\langle\zeta^3\right\rangle.
\end{equation*}
Since by assumption~$2p \ge p+2$, we thus have
\begin{equation*}
 \begin{split}
  \Delta_{p-2}(\zeta)
  & \equiv
  -2aA\zeta^{p-1}+A \left(3a^2-2b \right) \zeta^p +
A \left( -4a^3+6ab-2c \right) \zeta^{p+1}
  \\ & \quad
-  3Aa^2\zeta^p+3Aa\left(a^2 - b \right)\zeta^{p+1} \mod\left\langle\zeta^{p+2} \right\rangle
  \\ & \equiv
  -2aA\zeta^{p-1}-2Ab\zeta^p
- A \left( a^3-3ab+2c \right) \zeta^{p+1}
\mod\left\langle\zeta^{p+2} \right\rangle.
\end{split}
\end{equation*}
Then we have
\begin{equation*}
 \begin{split}
  \Delta_{p-1}(\zeta)
  & \equiv
  -2Aa\zeta^{p-1}\left[ \left( 1+a\zeta +b\zeta^2 +c\zeta^3 \right)^{p-1}-1\right] 
  \\ & \quad
  -2Ab\zeta^p\left[ \left( 1+a\zeta+b\zeta^2 \right)^p -1\right]
 \\ & \quad
- A \left( a^3 - 3ab + 2c \right)\zeta^{p+1} \left[ (1+a\zeta)^{p+1} -
  1 \right]
\mod\left\langle \zeta^{p+3}\right\rangle.
 \end{split}
\end{equation*}
Note that 
\begin{multline*}
  \left( 1+a\zeta + b\zeta^2+c\zeta^3 \right)^{p-1}-1
\equiv
-a\zeta+\left(a^2 - b \right)\zeta^2
+ \left( -a^3+2ab-c \right)\zeta^3
\\
\mod \left\langle\zeta^4\right\rangle.
\end{multline*}
Consequently, since by assumption~$2p \ge p+3$, we thus have
\begin{equation*}
 \begin{split}
  \Delta_{p-1}(\zeta)
  & \equiv
  2Aa^2\zeta^{p}-2Aa \left( a^2-b \right) \zeta^{p+1}
+ 2Aa \left( a^3-2ab+c \right)\zeta^{p+2}
  \\ & \quad
  - Aa \left( a^3 - 3ab + 2c \right)\zeta^{p+2}
\mod\left\langle \zeta^{p+3}\right\rangle
  \\ & \equiv
2Aa^2\zeta^{p}-2Aa\left(a^2 - b \right)\zeta^{p+1}
\\ & \quad
+Aa^2\left(a^2 - b \right)\zeta^{p+2}
\mod\left\langle \zeta^{p+3}\right\rangle.
  \end{split}
\end{equation*}
It follows that, using~$2p \ge p+4$, we have
\begin{equation*}
 \begin{split}
  \Delta_{p}(\zeta)
  & \equiv
2Aa^2\zeta^p\left[ \left( 1+a\zeta+b\zeta^2+c\zeta^3 \right)^{p} - 1 \right]
  \\  & \quad
- 2 A a \left(a^2 - b \right) \zeta^{p+1}\left[ \left(1+a\zeta +b\zeta^2\right)^{p+1}-1\right] 
  \\  & \quad
+ A a^2 \left(a^2 - b \right) \zeta^{p+2}\left[(1+a\zeta)^{p+2}
  -1\right]
\mod\left\langle \zeta^{p+4}\right\rangle
 \\ & \equiv
- 2 A a^2 \left(a^2 - b \right) \zeta^{p+2}
- 2 A ab \left(a^2 - b \right) \zeta^{p+3}
\\ & \quad
+ 2Aa^3 \left( a^2 - b \right) \zeta^{p + 3}
\mod\left\langle \zeta^{p+4}\right\rangle
\\ & \equiv
- 2 A a^2 \left(a^2 - b \right) \zeta^{p+2}
+ 2Aa \left(a^2 - b \right)^2 \zeta^{p + 3}
\mod\left\langle \zeta^{p+4}\right\rangle.
 \end{split}
\end{equation*}
Using~$\Delta_p(\zeta) = f^p(\zeta) - \zeta$, $A = - a^{p - 3}/2$, and
the definitions of~$\chi_1$ and~$\xi_1$, we obtain~\eqref{e:generic form} in the case where
$p\geq 5$ and~$n=1$.

\partn{Case 1.3}~$p\geq 3$ and~$n\geq 2$.
We proceed by induction.
Since~\eqref{e:generic form} for~$n = 1$ was shown in Cases~$1.1$ and~$1.2$, we can
assume that~\eqref{e:generic form} holds for some integer~$n \ge 1$.
Put
$$ g \= f^{p^n}, \text{ }
d \= \frac{p^{n+1}-1}{p-1},
\text{ }\alpha \= \chi_n,
\text{ and }
\beta \= \xi_n,
$$
so there is~$\gamma$ in~$k$ such that
$$
g(\zeta)
\equiv
\zeta+\alpha\zeta^{d+1}+\beta\zeta^{d+2}+\gamma\zeta^{d+3}
\mod \left\langle \zeta^{d + 4} \right\rangle.
$$
Note that
\begin{equation}
\label{e:main case recursion}
d \ge 4, \text{ }
d \equiv 1 \mod p, \text{ }
\chi_{n+1} = - \alpha^{p-1}\beta,
\text{ and }
\xi_{n+1} = -\alpha^{p-2}\beta^2.
\end{equation}
For each integer~$m\geq 1$ define the power series~$\hDelta_m(\zeta)$ in~$k[[\zeta]]$
inductively by~$\hDelta_1(\zeta)\=g(\zeta)-\zeta$, 
and for~$m\geq 2$ by
$$\hDelta_m(\zeta)\=\hDelta_{m-1}(g(\zeta))-\hDelta_{m-1}(\zeta).$$
Note that~$\hDelta_p(\zeta)=g^p(\zeta)-\zeta = f^{p^{n + 1}}(\zeta) - \zeta$.

We first prove
\begin{multline}
\label{e:delta_p-2ngeq2}
\hDelta_{p-2}(\zeta)
\equiv
\alpha^{p-2}\zeta^{(p-2)d+1}+\alpha^{p-3}\beta\zeta^{(p-2)d+2}+\alpha^{p-3}
\gamma\zeta^{(p-2)d+3}
\\
\mod\left\langle\zeta^{(p-2)d+4}\right\rangle.
\end{multline}
We then proceed to calculate~$\hDelta_{p-1}(\zeta)$ and~$\hDelta_{p}(\zeta)$ respectively.
Note that~\eqref{e:delta_p-2ngeq2} is true by definition when~$p=3$.
To prove~\eqref{e:delta_p-2ngeq2} for~$p\geq 5$, we first prove
\begin{equation}
\label{e:Delta_p-3 p5}
\hDelta_{p-3}(\zeta)\equiv -\frac{1}{2}\alpha^{p-3}\zeta^{(p-3)d+1}\mod\left\langle\zeta^{(p-3)d+3}\right\rangle.
\end{equation}
To prove this, for  integers~$m\geq1$ we define~$C_m$ in~$k$ inductively by  
\begin{equation*}
C_1\=1, \quad C_{m+1}\=\prod_{j=1}^{m}(jd+1)+(md+2)C_m.
\end{equation*}
We prove by induction that for~$m\geq 1$ we have
\begin{multline}
\label{e:Delta_m p5}
\hDelta_m(\zeta)
\equiv
\left( \prod_{j=0}^{m-1}(jd+1) \right) \alpha^m \zeta^{md+1} + C_m\alpha^{m-1}\beta\zeta^{md+2}
\\
\mod\left\langle\zeta^{md+3}\right\rangle.
\end{multline}
This follows by induction since it holds trivially for~$m=1$ and
for~$m \ge 1$ for which this holds we have
\begin{equation*}
 \begin{split}
   \hDelta_{m+1}(\zeta)
& \equiv
   \left (\prod_{j=1}^{m-1}(jd+1)\right ) \alpha^m\zeta^{md+1}
\left[\left( 1+\alpha\zeta^{d}+\beta\zeta^{d+1} \right)^{md+1}-1\right]
   \\ & \quad
+C_m\alpha^{m-1}\beta\zeta^{md+2}\left[ \left( 1+\alpha\zeta^{d}
  \right)^{md+2}-1\right]
\mod\left\langle\zeta^{(m+1)d+3}\right\rangle
   \\ & \equiv
 \left ( \prod_{j=1}^{m}(jd+1)\right ) \alpha^{m+1}\zeta^{(m+1)d+1}
   \\ & \quad
+ \alpha^{m}\beta\zeta^{(m+1)d+2} \left(\prod_{j=1}^{m}(jm+1)
  +(md+2)C_m\right)
\\ & \qquad
\mod\left\langle\zeta^{(m+1)d+3}\right\rangle.
 \end{split}
\end{equation*}
We then prove that~$C_{p-3}=0$. Using that~$d\equiv 1\mod p$, for~$m\geq 2$ we have that~$\{C_m\}_{m\geq 1}$
satisfies the recursive relation 
$$
C_{m+1}=(m+1)!+(m+2)C_{m},
$$
with initial condition~$C_1=1$. Hence~$\{C_{m}\}_{m\geq 1}$ satisfies~\eqref{e:beta recurrence} and thus~$C_{p-3}=0$
as shown in Case~$1.2$.
Moreover, using that for~$m=p-3$ we have~$(p-3)!=-1/2$ in~$k$, from~\eqref{e:Delta_m p5} 
we thus obtain~\eqref{e:Delta_p-3 p5} as asserted. 

The congruence~\eqref{e:Delta_p-3 p5} implies that for some~$A$ in~$k$
the series~$\hDelta_{p-3}(\zeta)$ satisfies
\begin{equation*}
\hDelta_{p-3}(\zeta)
\equiv
-\frac{1}{2}\alpha^{p-3}\zeta^{(p-3)d+1}+A\zeta^{(p-3)d+3}
\mod\left\langle\zeta^{(p-3)d+4}\right\rangle.
\end{equation*}
Using that~$d\geq 4$ and that~$p$ divides~$(p-3)d+3$,  we thus have
\begin{equation*}
 \begin{split}
   \hDelta_{p-2}(\zeta)
& \equiv
   -\frac{1}{2}\alpha^{p-3}\zeta^{(p-3)d+1}
   \left[ \left( 1+\alpha\zeta^{d}+\beta\zeta^{d+1}+\gamma\zeta^{d+2} \right)^{(p-3)d+1}-1\right]
\\ & \quad
+ A\zeta^{(p-3)d+3} \left[ \left(1 + \alpha \zeta^{d} \right)^{(p -3)d + 3} -
1 \right]
   \mod\left\langle\zeta^{(p-2)d+4}\right\rangle
   \\ & \equiv
   \alpha^{p-2}\zeta^{(p-2)d+1}+\alpha^{p-3}\beta\zeta^{(p-2)d+2}
\\ & \quad
+\alpha^{p-3} \gamma\zeta^{(p-2)d+3}
\mod\left\langle\zeta^{(p-2)d+4}\right\rangle.
 \end{split}
\end{equation*}
This completes the proof of~\eqref{e:delta_p-2ngeq2}.

Using that~\eqref{e:delta_p-2ngeq2} holds for all~$p\geq 3$, that~$d\geq 4$, and~$p$ divides~$(p-2)d+2$, we obtain
\begin{equation*}
 \begin{split}
\hDelta_{p-1}(\zeta)
& \equiv
\alpha^{p-2}\zeta^{(p-2)d+1}\left[ \left(
    1+\alpha\zeta^{d}+\beta\zeta^{d+1}+\gamma\zeta^{d+2} \right)^{(p-2)d+1}-1\right]
\\ & \quad
+ \alpha^{p-3}\beta \zeta^{(p-2)d +2}\left[ \left( 1+\alpha\zeta^d
+ \beta \zeta^{d + 1} \right)^{(p-2)d+2}-1\right]
\\ & \quad
+ \alpha^{p-3}\gamma \zeta^{(p-2)d +3}\left[ \left( 1+\alpha\zeta^d
  \right)^{(p-2)d+3}-1\right]
\mod\left\langle\zeta^{(p-1)d+4}\right\rangle
\\ & \equiv 
-\alpha^{p-1}\zeta^{(p-1)d+1}-\alpha^{p-2}\beta\zeta^{(p-1)d+2}
\\ & \quad
-\alpha^{p-2} \gamma\zeta^{(p-1)d+3}+\alpha^{p-2}
\gamma\zeta^{(p-1)d+3}
\mod\left\langle\zeta^{(p-1)d+4}\right\rangle
\\ & \equiv 
-\alpha^{p-1}\zeta^{(p-1)d+1}-\alpha^{p-2}\beta\zeta^{(p-1)d+2}
\mod\left\langle\zeta^{(p-1)d+4}\right\rangle.
 \end{split}
\end{equation*}
Using~\eqref{e:main case recursion}, we thus have
\begin{equation*}
 \begin{split}
   \hDelta_{p}(\zeta)
& \equiv
-\alpha^{p-1}\zeta^{(p-1)d+1}
\left[\left( 1+\alpha\zeta^d+\beta\zeta^{d+1} + \gamma \zeta^{d + 2} \right)^{(p-1)d+1}-1\right]
\\ & \quad
   -\alpha^{p-2}\beta\zeta^{(p-1)d+2}
\left[\left( 1+\alpha\zeta^d+\beta\zeta^{d+1} \right)^{(p-1)d+2}-1\right]
\\ & \qquad
\mod\left\langle\zeta^{pd+4}\right\rangle
   \\ & \equiv
   -\alpha^{p-1}\beta\zeta^{pd+2}-\alpha^{p-2}\beta^2\zeta^{pd+3}
\mod\left\langle\zeta^{pd+4}\right\rangle.
   \\ & \equiv
   \chi_{n + 1} \zeta^{pd+2}  + \xi_{n + 1} \zeta^{pd+3}
\mod\left\langle\zeta^{pd+4}\right\rangle.
 \end{split}
\end{equation*}
Using~$\Delta_p(\zeta) = f^{p^{n + 1}}(\zeta) - \zeta$ and the
definition of~$d$, this completes the proof of~\eqref{e:generic form} for~$p\geq 3$.

\partn{Case 2}~$p=2$.
We proceed by induction.
The case~$n = 1$ follows from 
the fact that for $f(\zeta)=\zeta +a\zeta^2+b\zeta^3+c\zeta^4+d\zeta^5+\dots$ in $k[[\zeta]]$, 
using repeatedly that $p=2$, we have
\begin{equation*}
\begin{split}
f^2(\zeta)
\equiv & 
 \zeta\left (1+a\zeta+b\zeta^2+c\zeta^3+d\zeta^4\right) +a\zeta^2\left(1+a\zeta +b\zeta^2+c\zeta^3\right)^2 \\
   & + b\zeta^3\left( 1+a\zeta+ b\zeta^2 \right)^3 +c\zeta^4\left(1+a\zeta \right)^4+d\zeta^5
\mod \langle\zeta^6\rangle\\
\equiv & \left(\zeta+a\zeta^2+b\zeta^3+c\zeta^4+d\zeta^5\right) +a\zeta^2\left(1+a^2\zeta^2\right) \\
   & + b\zeta^3\left(1+a\zeta+ b\zeta^2  \right)\left(1+a^2\zeta^2\right) +c\zeta^4+d\zeta^5
\mod \langle\zeta^6\rangle\\
\equiv &  \zeta+a(a^2-b)\zeta^4+b(a^2-b)\zeta^{5} \mod\langle\zeta^6\rangle.
\end{split}
\end{equation*}

Let~$n \ge 1$ be an integer for which~\eqref{e:generic form} holds,
and note that
\begin{equation}
  \label{e:recursive relation p = 2}
\chi_{n + 1} = \chi_n \cdot \xi_n
\text{ and }
\xi_{n + 1} = \xi_n^2.  
\end{equation}
Put
$$ \tDelta_1(\zeta)\=f^{2^n}(\zeta)-\zeta
\text{ and }
\tDelta_2(\zeta)
\=
\tDelta_1 \left( f^{2^n}(\zeta) \right)-\tDelta_1(\zeta), $$
and note that~$\tDelta_2(\zeta) = f^{2^{n + 1}}(\zeta) - \zeta$.
By our induction hypothesis there is~$A$ in~$k$ such that
$$
 \tDelta_{1}(\zeta)
\equiv
\chi_n \zeta^{2^{n+1}} + \xi_n \zeta^{2^{n+1}+1}
+
A \zeta^{2^{n+1}+2}
  \mod\left\langle\zeta^{2^{n+1}+3}\right\rangle,
$$
and therefore
\begin{multline*}
 \begin{aligned}
  \tDelta_{2}(\zeta)
 & \equiv
  \chi_n \zeta^{2^{n+1}}\left[ \left( 1 + \chi_n \zeta^{2^{n+1} - 1} +
      \xi_n \zeta^{2^{n+1}} + A \zeta^{2^{n+1}+1} \right)^{2^{n+1}} -1 \right]
  \\ & \quad
+ \xi_n \zeta^{2^{n+1}+1}\left[ \left( 1 + \chi_n \zeta^{2^{n+1} - 1} +
    \xi_n \zeta^{2^{n+1}} \right)^{2^{n+1} + 1} - 1 \right] 
 \end{aligned}
\\
  \mod\left\langle\zeta^{2^{n+2}+2}\right\rangle.
\end{multline*}
Note that
\begin{equation*}
\left( 1 + \chi_n \zeta^{2^{n+1} - 1} + \xi_n \zeta^{2^{n+1}}
+ A \zeta^{2^{n+1}+1} \right)^{2^{n+1}} - 1
\equiv 0\mod \left\langle \zeta^{2^{n+1}+2}\right\rangle,
\end{equation*}
and
\begin{multline*}
\left( 1 + \chi_n \zeta^{2^{n+1} - 1}
+ \xi_n \zeta^{2^{n+1}} \right)^{2^{n+1} + 1} - 1
\\ \equiv
\chi_n \zeta^{2^{n+1} - 1}  + \xi_n \zeta^{2^{n+1}}
\mod \left\langle \zeta^{2^{n+1}+1}\right\rangle.
\end{multline*}
We thus have by~\eqref{e:recursive relation p = 2}
\begin{equation*}
  \begin{split}
f^{2^{n + 1}}(\zeta) - \zeta
=
\tDelta_{2}(\zeta)
& \equiv
\chi_n \xi_n \zeta^{2^{n+2}}+ \xi_n^2 \zeta^{2^{n+2}+1} \mod\left\langle\zeta^{2^{n+2}+2}\right\rangle
\\ & \equiv
\chi_{n + 1} \zeta^{2^{n+2}}+ \xi_{n + 1} \zeta^{2^{n+2}+1} \mod\left\langle\zeta^{2^{n+2}+2}\right\rangle.
  \end{split}
\end{equation*}
This completes the proof of the induction step, and hence of the lemma
in the case~$p = 2$.

The proof of Lemma~\ref{l:generic form q=1} is thus complete.
\end{proof}

\begin{lemm}
\label{l:lth iterate q=1}
 Let~$p$ be a prime number, $k$~a field of characteristic~$p$,
and
$$ f(\zeta) = \zeta  + a\zeta^2+b\zeta^3+\cdots $$
a power series in~$k[[\zeta]]$.
Then for every integer~$\ell \geq 1$ we have
\begin{equation}
\label{e:lth iterate q=1}
f^{\ell}(\zeta)\equiv \zeta+\ell a\zeta^2 +(\ell (\ell-1)a^2+ \ell b)\zeta^3 \mod \left\langle \zeta ^4\right\rangle.
\end{equation}
\end{lemm}
\begin{proof}
We proceed by induction. When~$\ell =1$ the congruence~\eqref{e:lth iterate q=1}
holds by definition of~$f$. Let~$\ell\geq 1$ be an integer for which~\eqref{e:lth iterate q=1} holds.
Then
\begin{equation*}
\begin{split}
f^{\ell+1}(\zeta) 
& \equiv
\zeta+ \ell a\zeta^2 + \left( \ell (\ell-1)a^2+ \ell b \right) \zeta^3
 + a \left[ \zeta+ \ell a\zeta^2 \right]^2 + b\zeta^3
\mod \left\langle \zeta^4\right\rangle
\\ & \equiv
\zeta+ (\ell+1)a\zeta^2 + \left[ (\ell(\ell-1)+2\ell) a^2+(\ell +1)b \right] \zeta^3
\mod \left\langle \zeta^4\right\rangle
\\ & \equiv
\zeta+ (\ell+1)a\zeta^2 + \left[ (\ell+1)\ell a^2+(\ell +1)b \right] \zeta^3
\mod \left\langle \zeta^4\right\rangle.
\end{split}
\end{equation*}
This completes the induction step and hence the proof of~\eqref{e:lth iterate q=1} as required.
\end{proof}

\begin{proof}[Proof of the Main Lemma]
The case where~$q=1$ and $n=0$ is trivial and 
the case where~$q=1$ and $n\geq 1$ is given by Lemma~\ref{l:generic form q=1}. 
In what follows we assume~$q\geq 2$.
Put
$$ \pi(\zeta) \= \zeta^q
\text{ and }
\whg(\zeta) \= \zeta \left( 1 + \sum_{j = 1}^{+ \infty} a_j \zeta^j \right)^q, $$
and note that~$\pi \circ g = \whg \circ \pi$.
Clearly
\begin{equation*}
\whg(\zeta) 
\equiv
\zeta+ qa_1 \zeta^2 +\left( \frac{q(q-1)}{2}a_1^2+ q a_2  \right)\zeta^3   \mod \left\langle \zeta^4\right\rangle.
\end{equation*}
Using Lemma~\ref{l:lth iterate q=1} with~$\ell=q$,~$a=qa_1$, and~$b=\frac{q(q-1)}{2}a_1^2+qa_2,$ we obtain
\begin{equation}
\label{e:generic form g q iterates hat}
\begin{aligned}
\whg^{q}(\zeta) 
& \equiv
\zeta+ q (q a_1) \zeta^2
\\ & \quad
+ \left [q(q-1)(qa_1)^2 + q\left( \frac{q(q-1)}{2}a_1^2+ q a_2 \right ) \right]\zeta^3 
\mod \left\langle \zeta^4\right\rangle
\\ & \equiv
\zeta+ q^2 a_1\zeta^2
\\ & \quad
+q^2\left [\left (q(q-1)+\frac{q-1}{2}\right )a_1^2+ a_2   \right]\zeta^3 
 \mod \left\langle \zeta^4\right\rangle
\\ & \equiv
\zeta+ q^2 a_1\zeta^2 + q^2\left [\left (q^2-\frac{q+1}{2}
  \right)a_1^2+ a_2   \right] \zeta^3 
 \mod \left\langle \zeta^4\right\rangle.
\end{aligned}
\end{equation}
Let~$A_0$, $A_1$, \ldots, $A_q$ in~$k$ be such that
$$
g^{q}(\zeta)
\equiv
\zeta \left( 1 + A_0\zeta^q + \cdots + A_q \zeta^{2q} \right)
\mod \left\langle \zeta^{2q + 2} \right\rangle.
$$
Then we have
\begin{multline*}
\pi \circ g^q(\zeta)
\equiv
\zeta^q \left( 1 + q A_0\zeta^q + \cdots + q A_{q - 1} \zeta^{2q - 1} + \left( \frac{q(q - 1)}{2}
    A_0^2 + q A_q \right) \zeta^{2q} \right)
\\
\mod \left\langle \zeta^{3q + 1} \right\rangle.
\end{multline*}
Together with the semi-conjugacy~$\pi \circ g^{q} = \whg^{q} \circ
\pi$ and with~\eqref{e:generic form g q iterates hat}, this implies
$$ A_0 = q a_1,
A_1 = \cdots = A_{q - 1} = 0,
\text{ and }
A_q = q \left[ \left(\frac{q^2 - 1}{2} \right)a_1^2+ a_2 \right]. $$
This proves the desired assertion when~$n=0$.  

Fix an integer~$n \ge 1$.
By the semi-conjugacy~$\pi \circ g^{q} = \whg^{q} \circ \pi$, we have~$\pi \circ g^{qp^n} = \whg^{qp^n} \circ \pi$.
On the other hand, by~\eqref{e:generic form g q iterates hat}, the
definitions of~$\chi_{q, n}$ and~$\xi_{q, n}$, the fact that~$q^p = q$
in~$k$, and by Lemma~\ref{l:generic form q=1} with
\begin{equation*}
a = q^2a_1
\text{ and }
b = q^2\left (q^2-\frac{q+1}{2} \right)a_1^2+ q^2a_2,
\end{equation*}
we have, using
\begin{equation*}
a^2 - b
=
q^2\left (\frac{q+1}{2}a_1^2 - a_2\right),
\end{equation*}
that
\begin{multline}
\label{e:generic form g iterates hat}
\whg^{qp^n}(\zeta)
\equiv
\zeta \left( 1 + q \chi_{q, n} \zeta^{\frac{p^{n+1}-1}{p-1}} + q
  \xi_{q, n} \zeta^{\frac{p^{n+1}-1}{p-1}+1} \right)
\\
\mod \left\langle \zeta^{\frac{p^{n+1}-1}{p-1}+3} \right\rangle.
\end{multline}
Together with~$\pi \circ g^{qp^n} = \whg^{qp^n} \circ \pi$, this implies
that there are constants~$B_0$, $B_1$, \ldots, $B_q$ in~$k$ such that
\begin{multline*}
g^{qp^n}(\zeta)
\equiv
\zeta \left( 1+B_0\zeta^{q\frac{p^{n+1}-1}{p-1}}+ \cdots + B_q
  \zeta^{q\frac{p^{n+1}-1}{p-1}+q} \right)
\\
\mod \left\langle \zeta^{q\frac{p^{n+1}-1}{p-1}+ q+2} \right\rangle.
\end{multline*}
Consequently, we have
\begin{multline*}
\pi\circ g^{qp^n}(\zeta)
\equiv
\zeta^q \left( 1 + qB_0 \zeta^{q\frac{p^{n+1}-1}{p-1}}
+ \cdots + 
qB_q\zeta^{q\frac{p^{n+1}-1}{p-1}+q} \right) 
\\
\mod \left\langle \zeta^{q\frac{p^{n+1}-1}{p-1}+ q+2} \right\rangle.
\end{multline*}
Using~$\pi \circ g^{qp^n} = \whg^{qp^n} \circ \pi$
and~\eqref{e:generic form g iterates hat}, it follows that
$$ B_0 = \chi_{q, n},
B_1 = \cdots = B_{q - 1} = 0,
\text{ and }
B_q = \xi_{q, n}. $$
This completes the proof of the Main Lemma.
\end{proof}

\addcontentsline{toc}{section}{References}

\bibliographystyle{alpha}

\end{document}